\documentclass[a4paper]{amsart}

\usepackage{amsmath}
\usepackage{amsthm}
\usepackage{txfonts}
\usepackage{graphicx}
\usepackage{url}
\usepackage{color}

\theoremstyle{theorem}

\newtheorem{theorem}{Theorem}[section]

\newtheorem*{mq}{Problem}

\newtheorem{corollary}[theorem]{Corollary}

\newtheorem{lemma}[theorem]{Lemma}

\newtheorem{remark}[theorem]{Remark}

\theoremstyle{definition}
\newtheorem{definition}[theorem]{Definition}
\newtheorem*{convention}{Convention (C)}

\numberwithin{equation}{section}

\DeclareMathOperator{\interior}{int}
\DeclareMathOperator{\diam}{diam}%
\newcommand{\edges}[1]{{#1}^e}

\DeclareMathOperator*{\St}{St}
\DeclareMathOperator*{\val}{val}
\DeclareMathOperator{\Nacc}{\mathcal{I}}
\DeclareMathOperator{\NaccSides}{\mathcal{IS}}%
\DeclareMathOperator{\disc}{D^s}%
\DeclareMathOperator{\nadler}{\Lambda}
\DeclareMathOperator{\euler}{\chi} 
\DeclareMathOperator{\ball}{B}%

\newcommand{\defmapsto}{\rightarrow}
\newcommand{\pure}[1]{\mathcal{P}_{#1}}
\newcommand{\mix}{\mathcal{M}}
\newcommand{\exact}{\mathcal{E}}

\newcommand{\Lom}{\mathcal{L}}

\newcommand{\cl}[1]{\overline{#1}}
\newcommand{\set}[1]{\left\{#1\right\}}
\newcommand{\eps}{\varepsilon}

\newcommand{\ra}{\to}

\makeatletter
\def\@strippedMR{}
\def\@scanforMR#1#2#3\endscan{%
  \ifx#1M\ifx#2R\def\@strippedMR{#3}%
  \else\def\@strippedMR{#1#2#3}%
  \fi\fi}
\renewcommand\MR[1]{\relax
  \ifhmode\unskip\spacefactor3000 \space\fi
  \@scanforMR#1\endscan
  MR\MRhref{\@strippedMR}{\@strippedMR}}
\makeatother

\begin{document}
\title[Structure and entropy of mixing maps]{Topological structure and entropy of mixing graph maps}


\author[Hara\'{n}czyk]{Grzegorz Hara\'{n}czyk}
\address[G.~Hara\'{n}czyk]{Institute of Mathematics, Jagiellonian University,
{\L}ojasiewicza~6, 30-348 Krak\'ow, Poland} \email{gharanczyk@gmail.com}

\author[Kwietniak]{Dominik Kwietniak}
\address[D.~Kwietniak]{Institute of Mathematics, Jagiellonian University,
{\L}ojasiewicza~6, 30-348 Krak\'ow, Poland} \email{dominik.kwietniak@uj.edu.pl}

\author[Oprocha]{Piotr Oprocha}

\address[P.~Oprocha]{ Faculty of Applied
Mathematics, AGH University of Science and Technology, al.
Mickiewicza 30, 30-059 Krak\'ow, Poland -- and -- Institute of Mathematics, Polish Academy of Sciences, ul. \'Sniadeckich 8, 00-956 Warszawa, Poland} \email{oprocha@agh.edu.pl }

\subjclass[2000]{Primary: 37B40, 37B20 Secondary: 37E05, 37E10}

\date{\today}

\dedicatory{}

\begin{abstract}
Let $\pure{G}$ be the family of all topologically mixing,
but not exact self-maps of a topological graph $G$. It is
proved that the infimum of topological entropies of maps
from $\pure{G}$ is bounded from below by $(\log 3/ \nadler(G))$,
where $\nadler(G)$ is a constant depending on the
combinatorial structure of $G$.
The exact value of the
infimum on $\pure{G}$ is calculated for some families of graphs.
The main tool is a refined version of the structure theorem
for mixing graph maps. It also yields new proofs of some
known results, including Blokh's theorem
(topological mixing implies specification property for maps
on graphs).
\end{abstract}

\maketitle

\section{Introduction}
There is no connection between topological transitivity,
and topological entropy for self-maps of general compact
metric spaces. A map with positive entropy need not to be
transitive, and a transitive system may have zero entropy.
However, there are spaces such that every topologically
transitive map on them must have positive topological
entropy. For instance, on the compact  interval $[0,1]$ every
transitive map has entropy at least $\log\sqrt{2}$, and
there is a transitive self-map of $[0,1]$ with entropy
equal to this bound. The similar questions about the
best lower bounds for the topological entropy in various
classes of transitive self-maps of a  fixed space have
been considered by many authors, see
\cite{ABLM,AKLS,ALM,Baldwin,BS,Dirbak,HK,KM,Ye}.
For more references and other results of this type, e.g.,
lists of known best lower bounds for the topological
entropy of transitive maps on various spaces, see
\cite[page 341]{ALM} or \cite{AKLS,BS,HK,KM}.
The present work is motivated by the following problem:

\begin{mq}\label{main-problem}
Let $G$ be a topological graph. Let $\pure{G}$ denote the family
of all pure mixing (that is, topologically mixing, but not
topologically exact) self-maps of $G$. Find the infimum of
topological entropies of maps from $\pure{G}$, which is
hereafter denoted as $\inf h(\pure{G})$.
\end{mq}

The main result we would like to present here states that
for a given graph $G$ we have
\[
\frac{\log 3}{\nadler(G)}\le \inf h(\pure{G}),
\]
where $\nadler(G)$ is a constant depending on the combinatorial
structure of the graph $G$. Moreover, we are able to compute
$\inf h(\pure{G})$ for some graphs and two infinite
families of trees (defined in terms of some special structural properties).

This is a generalization of results from \cite{HK}, where the
pure mixing maps of the interval and the circle were considered.
But it should be stressed that the methods from \cite{HK} can
not be directly adapted to the more general case considered here.
More precisely, they can be (after some modification) used to prove the similar results for
trees, but are not applicable for graphs containing a circle and
at least one branching point. The reasons are twofold: first it
is harder to obtain a covering relation from containment relation
if the graph
contains a circle, second the interior of a connected set is no
longer connected if the graph contains a branching point. The
proofs in \cite{HK} heavily relies on these two facts.

To solve our problem in the new, more general context, we refine
the structure theorem for pure mixing graph maps
(see Theorem~\ref{thm:structure}), and apart of
the estimate for the topological entropy mentioned above, we obtain
(with some additional work) new proofs of two results that are
of general interest: Blokh's theorem, stating that
topological mixing graph maps have the specification property
(see Theorem~\ref{thm:Blokh}),
and \cite[Theorem 4.2]{Yokoi}, the main result of \cite{Yokoi},
which in turn, generalizes \cite[Theorem B and C]{CovenMulvey} to graph maps (see
Theorem \ref{thm:Yokoi-main}).

Finally, we would like to remark that our version of the
structure theorem (Theorem~\ref{thm:structure}) for pure mixing graph maps could probably be
derived along the lines of Blokh's papers (see also the
presentation of Blokh's work in Alseda et al. \cite{Alseda}),
but it is not a simple nor direct corollary of any theorem
presented in \cite{Blokh, Blokh2, Blokh3}. To obtain the structure theorem
from \cite{Blokh, Blokh2, Blokh3} or \cite{Alseda}, one should rather rework the whole proof,
and adjust it in many places. We are convinced that this approach to
the structure theorem, even if succeeded,  would result in less transparent and longer
proof than ours. 

\subsection{Entropy and chaos for graph maps}

Let us briefly recall one of the possible interpretations of
the lower bound for topological entropy in
the class $\pure{G}$. First, note that for any fixed
graph $G$, which is not a tree (contains at least one circle),
the infimum of topological entropies of mixing maps is zero
(see \cite{Alseda}).

Now, our theorem about lower bound for the topological entropy
of pure mixing graph map can be rephrased in a following way.
Let $G$ be any graph containing a circle. If we add
(set theoretically) the family $\exact(G)$ of all topologically
exact maps of $G$ to the family $\pure{G}$ of all pure mixing
maps then we get the family $\mix(G)$ of all topologically
mixing maps. As we observed above in $\mix(G)$ we can find maps
with arbitrary low entropy. By our result, if the entropy of
a mixing graph map of $G$ is sufficiently small, then the map
must be exact, that is, exact maps lower the entropy
in the family of mixing maps.
On the other hand, exact maps are regarded as more chaotic
than pure mixing maps. Therefore, we can once again
re-formulate our main result:
adding \emph{more chaotic} class of maps to the
\emph{less chaotic} one results in \emph{lowering}
topological entropy in the enlarged class.
This contrasts with the common interpretation of the entropy
as a quantitative measure of chaos present in the system.

But the paradox disappears if only we will treat the topological
entropy as a qualitative indicator of chaos, that is, positive
topological entropy is a sign of complex behavior present in
the system. From this point of view the precise numerical value
of the topological entropy is unimportant.

\section{Basic definitions and notation}

\subsection{Notation and terminology}
Let $(X,\rho)$ be a metric space, and let $f\colon
X\defmapsto X$ be a continuous map. In this paper letters $k,l,m,n, M, N$ will always
denote integers, and by ``a map'' we will always
mean ``a continuous map''. If $A\subset X$ then
we will denote the closure (interior) of $A$ by
$\cl{A}$ ($\interior{A}$, respectively). An open ball
with the center at $x\in X$ and radius $\eps>0$ is
denoted $\ball(x,\eps)$. Similarly, if $A$ is a subset
of $X$ then $\ball(A,\eps)=\bigcup_{x\in A}\ball(x,\eps)$.

A \emph{continuum} is a compact, connected metric space.
An \emph{arc} is a continuum homeomorphic to the interval
$[0,1]$. If $A$ is an arc and $g\colon [0,1]\defmapsto A$
is a homeomorphism, then the \emph{endpoints} of $A$
are $g(0)$ and $g(1)$. Clearly, the endpoints do not
depend on the choice of $g$.

If $X$ is a continuum and $A\subset X$ is an arc with endpoints
$x$ and $y$, then we say that $A$ is a \emph{free arc} provided $A\setminus\{x,y\}$ is open in $X$.

\subsection{Topological graphs}
A \emph{topological graph} (a \emph{graph} for short)
is a continuum $G$ such that there is an one dimensional
simplical complex $\mathcal{K}$ whose geometric carrier
$|\mathcal{K}|$ is homeomorphic to $G$. Each such complex
is called a \emph{triangulation} of $G$.
We say that a triangulation $\mathcal{L}$ of a graph $G$ is a
\emph{subdivision} of a triangulation $\mathcal{K}$ if every vertex
$\mathcal{L}$ is a vertex for $\mathcal{K}$.
We identify each graph $G$ with a subspace of the Euclidean space
$\mathbb{R}^3$. Moreover, we assume that $G$ is endowed with the
taxicab metric, that is, the distance between any two points of $G$
is equal to the length of the shortest arc in $G$ joining these
points. If $G$ is a graph, and $\mathcal{K}$ is a triangulation of
$G$, then every zero (one) dimensional simplex of $\mathcal{K}$
is called a \emph{vertex} (an \emph{edge})
of $G$ with respect to $\mathcal{K}$.
The set of all edges with respect to $\mathcal{K}$ is denoted by
$\edges{\mathcal{K}}$.

The \emph{star} of a vertex $v$, denoted
by $\St(v)$, is the union
of all the edges that contain the vertex $v$.
For every $x\in G$ we define the \emph{valence} of $x$, denoted
$\val(x)$, in the following way: if $x$ is an vertex of $G$ then
$\val(x)$ is equal to the number of connected components
of $\St(v)\setminus\{v\}$, and $\val(x)=2$ otherwise.
Points $x\in G$ with $\val(x)=1$ are called
\emph{endpoints} of the graph $G$, and if $\val(x)>2$ we say that
$x$ is a \emph{branching point}.
Let $End(G)$ denote the set of all endpoints of $G$.


Note that $\val(x)$ is independent of the choice of
triangulation. In particular, every branching point (endpoint) for some
triangulation is a branching point (endpoint) for every
triangulation.

Any subset of $G$ which is a graph itself is called a
\emph{subgraph} of $G$. The family of all subgraphs of $G$ is
denoted by $\mathcal{G}(G)$, and coincides with the family of
all nondegenerate subcontinua of $G$.
Note that a singleton set is \emph{not} a subgraph of $G$
(graphs are one-dimensional by the definition).

Following Nadler \cite{Nadler}, we define \emph{disconnecting
number} of a graph $G$ as the least $n>0$ such that every
subset $D$ of $G$ of cardinality $n$ \emph{disconnects} $G$
(i.e., $G\setminus D$ is not connected). The disconnecting number
is well defined and is denoted $\disc(G)$ (see \cite{Nadler}).
Let $\nadler(G)$ be the maximal disconnecting number among all
subgraphs of $G$. It follows from \cite{Nadler} that
$\nadler(G)=\disc(G)-\euler(G)$+1, where $\euler(G)$ is the
Euler characteristic of $G$.

\subsection{Topological dynamics}

We refer the reader to \cite{ALM} for definitions of basic
concepts of the theory of dynamical systems, such as
\emph{(periodic) orbit}, \emph{(semi-) conjugacy} etc. We call a
map $f$ \emph{transitive} if for every pair of nonempty open
subsets $U$ and $V$ of $X$ there is an $n$ such that $f^n(U)\cap
V\neq\emptyset$; we say that $f$ is \emph{totally transitive} if
all its iterates $f^n$ are transitive; 
a map $f$ is \emph{mixing} if for any nonempty sets $U$ and $V$ open in
$X$, there is an $N>0$ such that $f^n(U)\cap V\neq\emptyset$ for
$n \ge N$; a map $f$ is \emph{exact} if for any nonempty open
subset $U$ of $X$ there is an $n\ge 0$ such that $f^n(U)=X$. It is well
known that exactness implies mixing, and mixing
implies total transitivity, but not conversely in general.
In the special case of non-invertible graph maps total transitivity
implies mixing (see \cite{HKO} for a simple proof of that fact).
The only examples of transitive graph homeomorphisms
are the irrational rotations of the circle, which are even totally transitive,
but not mixing.
Moreover, the transitive graph maps are either totally transitive, or
can be decomposed into totally transitive
ones.  The precise statement is presented below and its proof can be found
in \cite[Theorem~2.2]{AlsedaSplit} (with the only difference that only
transitivity instead of total transitivity is claimed in \eqref{con:split:tt},
however the stronger conclusion follows easily from the proof presented there).
Alternatively, it follows from Banks periodic decomposition theorem \cite{BanksPD}.

\begin{theorem}\label{thm:trans-not-tot-trans}
Let $f\colon G \defmapsto G$ be a transitive graph map. Then exactly one
of the following two statements holds:
\begin{enumerate}
\item $f$ is totally transitive,
\item There exist a $k>1$ and
non-degenerate connected subgraphs $G_0,\ldots, G_{k-1}$ of $G$
such that
\begin{enumerate}
\item $G=\bigcup_{i=0}^{k-1}G_i$,
\item $G_i\cap G_j=End(G_i)\cap
End (G_j)$ for $i \neq j$,
\item $f(G_i)=G_{(i+1 \mod k)}$ for
$i=0,\ldots, k-1$,
\item\label{con:split:tt} $f^k|_{G_i}$ is totally transitive for
$i=0,\ldots, k-1$.
\end{enumerate}
\end{enumerate}
\end{theorem}

We say that a map $f$ is \emph{pure mixing} if $f$ is
mixing but not exact. We recommend \cite{KSSurvey} as a source of
information on transitivity.

For a definition of the \emph{topological entropy} of $f$ we refer the
reader again to \cite{ALM}. Recall that, if $X$ is a compact space then
the entropy of $f$ is a (possibly infinite) number $h(f)\in [0,+\infty]$.
We will use the basic properties of the entropy such as those in
\cite[Section 4.1]{ALM} without further reference.

Let $\mathfrak{F}(X)$ be a subclass of the class of transitive self-maps
of a given compact metric space $X$. By $\inf h( \mathfrak{F}(X))$ we
mean the best lower bound for the topological entropy of maps
from $\mathfrak{F}(X)$, that is, $\inf h(\mathfrak{F}(X))=\inf \{h(f): f\in
\mathfrak{F}(X)\}$. Moreover, we say that $\inf h(\mathfrak{F}(X))$ is
\emph{attainable} if there exists a map $f\in\mathfrak{F}(X)$ such that
$h(f)=\inf h(\mathfrak{F}(X))$.

\section{Some properties of graph maps}

In this section we collect some properties of graph maps, which
we will use frequently in subsequent sections.
The following convention will recur frequently in what follows.

\begin{convention}\label{convention}
Let $J$ be a free arc (e.g., an edge) in a graph $G$, and let $e$ be one
of its endpoints. We identify $J$ with an interval $[0,a]\subset\mathbb{R}$,
where $0<a$ and $e$ is identified with $0$. We may also assume that this
identification is actually an isometry if necessary, in particular
$a=\diam J$. Then $J$ could be linearly ordered by the relation $\le$
induced from $[0,a]$. It allows us to write
$x<y$ to denote the relative position of points on $J$, and use usual
interval notation to describe connected subsets of $J$.
\end{convention}

The proof of the next lemma is omitted, as it is straightforward.
Alternatively, it can be deduced from \cite[Lemma~23]{KO_CSF} or
\cite[Theorem~3.11]{Kato}.

\begin{lemma}\label{lem:mixing_graph}
For a map $f$ of a graph $G$ the following conditions are equivalent:
\begin{enumerate}
\item $f$ is mixing,\label{con:mixing_simple}
\item\label{con:mixing_graph_unifrom_component}
for every $\eps>0$ and $\delta>0$ there is an integer $N=N(\eps,\delta)$
such that for any subgraph $J$ of $G$ with $\diam J\ge \delta$ each connected
component of the set $G\setminus f^n(J)$ has diameter smaller then
$\eps$ for every $n\ge N$.
\end{enumerate}
\end{lemma}

\begin{definition}
Given a graph map $f$ and free arcs $I,J\subset G$ we
say that $I$ \emph{covers} $J$ through $f$ (or $f$-covers,
for short) if there exists a free arc $K \subset I$ such
that $f (K) = J$.
\end{definition}

Properties of $f$-covering relation
presented below have elementary proofs. The first
five of them are adapted from \cite[p. 590]{Alseda}
(note that closed intervals there are arcs in our
terminology).

\begin{lemma}\label{lem:covering}
Let $I,J,K,L\subset G$ be free arcs, and let $f, g\colon G\defmapsto G$
be graph maps.
\begin{enumerate}
\item \label{lem:covering:Al1}
If $I$ $f$-covers $I$, then there exists $x\in I$ such that $f(x)=x$.
\item \label{lem:covering:Al2}
If $I\subset K$, $L\subset J$ and $J$ is $f$-covered by $I$, then $K$ $f$-covers $L$.
\item \label{lem:covering:Al3}
If $I$ $f$-covers $J$ and $J$ $g$-covers $K$, then $I$ $(g\circ f)$-covers $K$.
\item If $J\subset f(I)$, and $K_1,K_2\subset J$ are free arcs such that
$K_1\cap K_2$ is at most one point, then $I$ $f$-covers $K_1$, or $I$ $f$-covers $K_2$.\label{lem:covering:Al4}
\item \label{lem:covering:Al5}
If $J\subset f(I)$ is a free arc, then there exist free arcs $J_1,J_2$
such that $\interior J_1\cap \interior J_2=\emptyset$, $J_1\cup J_2=J$, and
$J_1,J_2$ are $f$-covered by $I$.
\item \label{lem:covering-one-endpoint}
If $K=J\cap f(I)$ contains at most one endpoint of $J$, then $K$ is $f$-covered by $I$.
\item \label{lem:covering-star-image}
If $S\subset G$ is a star and $J\subset f(S)$ then there are two free arcs
$E_1,E_2\subset S$ with at most one common point such that $J$ is contained in
the sum of their images, equivalently, $J\subset f(E_1\cup E_2)$.
\end{enumerate}
\end{lemma}

\begin{lemma}\label{lem:multiple_covering}
Suppose $Z=C_1\cup\ldots\cup C_n$, where $n\ge 2$ and $C_1,\ldots,C_n$ are
pairwise disjoint free arcs contained in the interior of a free arc
$F\subset G$. Let $J$ be a free arc in $G$ such that
$f(J)$ intersects the interior of any connected component of $F\setminus Z$.
Then $J$ $f$-covers at least $n-1$ sets among $C_1,\ldots,C_{n}$.
\end{lemma}
\begin{proof}
Note that at most one among sets $C_1,\ldots,C_n$ is not contained in
$f(J)$. Let $F'$ denote the convex hull of $C_1,\ldots,C_n$ in $F$,
that is, the intersection of all compact connected sets
containing $C_1\cup\ldots\cup C_n$. Clearly, $F'$ is a free arc.
There are two cases to consider. First, assume that $f(J)$ contains
only one endpoint of $F'$. Then use
Lemma~\ref{lem:covering}\eqref{lem:covering-one-endpoint}
and Lemma~\ref{lem:covering}\eqref{lem:covering:Al2} to see that all
$C_i$'s except at most one are $f$ covered by $J$.
In the second case there are free arcs $K$ and $L$ such that
$K\cup L = F'\cap f(J)$ and $K$ and $L$ have at most one common
point. Apply Lemma~\ref{lem:covering}\eqref{lem:covering:Al5} to see
that $K$ and $L$ must be $f$-covered by $J$, and observe that at most
one among sets $C_1,\ldots,C_n$ is not contained in $K\cup L$.
\end{proof}

The proof of the next lemma is left to the reader.

\begin{lemma}\label{lem:geometric3}
For each $\delta>0$ there is a constant $\xi=\xi(\delta)$ such that if
$K$ is a connected subset of $G$ with $\diam K\ge \delta$ then $K$
contains a free arc $J$ with $\diam J\ge \xi$.
\end{lemma}

%

\section{Global behavior of mixing graph maps}

By definition, a mixing map $f\colon G\defmapsto G$ is pure mixing
if and only if there exists an open set $U\subset G$ such that
$f^n(U)\neq G$ for all $n\ge 0$. We are going to prove that pure
mixing of a graph map is equivalent to the existence of a special
set of \emph{inaccessible points} $\Nacc(f)$, which are not
contained in $\interior f^n(U)$ for all $n\ge 0$ for any
open set $U\subset G$ disjoint from $\Nacc(f)$. The set of
inaccessible points has cardinality bounded above by $\disc(G)$,
is forward invariant, $f(\Nacc(f))\subset \Nacc(f)$, and all its
points are periodic points of $f$. The result is implicit in
\cite{Blokh,Alseda}. But our method of proof is new, and we prove
in addition that for a given open set $U\subset G$ the sets
$\interior f^n(U)$ grow in $G$ as $n\to\infty$. Moreover,
if $x\in\interior f^n(U)$ for some $n\ge 0$, then eventually
$x\in\interior f^k(U)$ for all sufficiently large $k$.

\begin{definition}
We say that a free arc $I_U\subset G$ is an \emph{universal arc} for a map $f$
of $G$ if for any $\delta>0$ there is an integer $M=M(\delta)$ such
that $I_U$ is $f^m$-covered by any free arc $J$ of $G$ with
$\diam J\ge \delta$ for all $m\ge M$.
\end{definition}
\begin{lemma}\label{lem:selfcover}
Let $f\colon G\defmapsto G$ be a mixing map. For every free arc
$F\subset G$ there is an universal arc $I_U\subset F$. Moreover,
\begin{enumerate}
\item\label{con:self-covering} there exists an integer $N_U>0$
such that $I_U$ is $f^n$-covered by itself for all $n\ge N_U$,
in particular
$\interior I_U\cup \interior f(I_U)\cup\ldots \cup \interior f^k(I_U) \subset \interior f^{N_U+k}(I_U)$
for all $k\ge 0$.
\item for any subgraph $J$ of $G$ we have
\[
G\setminus \bigcup _{j=0}^\infty  \interior f^j(J)\subset G\setminus \bigcup _{j=0}^\infty  \interior f^j(I_U).\label{con:universal-catching}
\]
\end{enumerate}
\end{lemma}
\begin{proof}Let $F$ be any free arc in $G$.
According to our convention~\textbf{(C)} we may identify $F$ with $[0,7]$. We define
\[
A=[1,2],\quad B=[3,4],\quad C=[5,6],\quad \text{and}\quad I_j=(2j,2j+1),\quad \text{for}\;j=0,\ldots,3.
\]
Since $f$ is mixing we can find an integer $k>0$ such that for any $E\in\{A,B,C\}$ we
have $f^k(E)\cap I_j\neq \emptyset$ for any $0\le j \le 3$ . Using
Lemma~\ref{lem:multiple_covering} with $Z=A\cup B\cup C$ we deduce that $A$ and $C$
must $f^k$-cover at least two intervals among $A,B,C$. It follows that there is
$I_U\in\{A,B,C\}$ which is $f^k$-covered by both, $A$ and $C$. We will show that
$I_U$ is an universal arc. Fix $\delta>0$ and some closed interval $J$ of $G$
with $\diam J\ge \delta$. Let $\eps=\max\{\diam I_j:0\le j\le 1\}$.
Lemma~\ref{lem:mixing_graph}\eqref{con:mixing_graph_unifrom_component} gives us
$N=N(\eps,\delta)$ such that for every $n\ge N$ each connected component of
$G\setminus f^n(J)$ has diameter less than $\eps$. We conclude that $f^n(J)$ must
intersect every connected component of $F\setminus (A\cup B\cup C)$, hence
Lemma~\ref{lem:multiple_covering} guarantees that for any $n\ge N$ some $D_n\in\{A,C\}$
must be $f^n$-covered by $J$. By the above, and Lemma~\ref{lem:covering}\eqref{lem:covering:Al3},
the free arc $I_U$ is $f^n$-covered by $J$ for any $n\ge N+k$. Therefore, we set
$M(\delta)=N+k$, and $I_U$ is an universal arc for $f$ as claimed. It follows
immediately that \eqref{con:self-covering} holds with $N_U=N(\diam I_U,\eps)+k$.
For the proof of \eqref{con:universal-catching} we fix subgraph $J$ of $G$, and we
let $\delta=\diam J$. Lemma~\ref{lem:geometric3} and the definition of $I_U$ above,
give us $M=M(\xi(\delta))$ such that $f^j(I_U)\subset f^{M+j}(J)$ for all $j\ge 0$,
and so
\[
\bigcup _{j=0}^\infty  \interior f^j(I_U)\subset \bigcup _{j=M}^\infty  \interior f^j(J)\subset \bigcup _{j=0}^\infty  \interior f^j(J).
\]
Taking complements we have
\begin{equation}\label{eq:two}
G\setminus \bigcup _{j=0}^\infty  \interior f^j(J)\subset G\setminus \bigcup _{j=0}^\infty  \interior f^j(I_U),
\end{equation}
as desired.
\end{proof}

It can be shown that a transitive graph map is mixing if and only if it has
an universal arc, but we will not pursue this here. We are now in a position
to define inaccessible points and characterize them with the help of the
interiors of images of the universal arc.

\begin{definition}
Let $f$ be a map of $G$, and let $\mathcal{G}$ denote the family of all
subgraphs of $G$. We define the set of \emph{inaccessible} points of $f$ by
\[
\Nacc(f)=
G\setminus \bigcap_{J\in\mathcal{G}} \bigcup_{k=0}^\infty \interior f^k (J) =
\bigcup_{J\in\mathcal{G}} \bigcap_{k=0}^\infty G\setminus \interior f^k (J),
\]
where the second equality above follows easily from elementary properties of
operations involved.
\end{definition}

\begin{lemma}\label{lem:universal-int-gives-nacc}
If $f$ is a mixing graph map and $I_U$ is an universal arc for $f$ then
\[
\Nacc(f)=G \setminus \bigcup _{k=0}^\infty \interior f^k(I_U).
\]
\end{lemma}
\begin{proof}
It is enough to observe that
$$
\Nacc(f)=\bigcup_{J\in\mathcal{G}} G \setminus \bigcup _{k=0}^\infty \interior f^k(J),
$$
and next apply Lemma~\ref{lem:selfcover}\eqref{con:universal-catching}.
\end{proof}
\begin{lemma}\label{lem:universal-universal-arc}
Let $f$ be a mixing graph map. A free arc is an universal arc for $f^l$ for some
$l\ge 1$ if and only if it is an universal arc for $f^n$ for all $n\ge 1$.
\end{lemma}
\begin{proof}It is clear that an universal arc for $f$ is also universal
for $f^l$ for any $l\ge 2$. It is now enough to show that an universal arc
$I_U^l$ for $f^l$, where $l\ge 2$ is also universal for $f$. To this end,
fix $\delta>0$, let $I_U$ be an universal arc for $f$, and let $M=M(\delta)$
be such that every free arc $J\subset G$ with $\diam J\ge \delta$ covers $I_U$
through $f^m$ for all $m\ge M$. We
can also find an $L\ge 0$ (a multiple of $l$) such that $I_U^l$ is $f^L$
covered by $I_U$. Now, every free arc $J\subset G$ with $\diam J\ge \delta$
covers $I_U^l$ through $f^m$ for every $m\ge M+L$, and the lemma follows.
\end{proof}

\begin{theorem}\label{thm:nacc}
For each mixing map $f$ on a graph $G$ the set $\Nacc(f)$ has the following properties:
\begin{enumerate}
\item For any $\delta,\eps>0$ there is an integer $K=K(\eps,\delta)$ such that for any
subgraph $J$ of $G$ with $\diam J\ge \delta$ we have
$G\setminus \interior f^{k}(J) \subset \ball(\Nacc(f),\eps)$ for all $k\ge K$.\label{con:nacc1}
\item The set $\Nacc(f)$ has less than $\disc(G)$ elements.
Moreover, $\Nacc(f)\neq\emptyset$ if and only if $f$ is pure mixing.\label{con:nacc2}
\item Each point $x\in\Nacc(f)$ is periodic for $f$ and $f(\Nacc(f))=\Nacc(f)$.\label{con:nacc3}
\item For every $n\ge 1$ we have $\Nacc(f)=\Nacc(f^n)$.\label{con:nacc4}
\end{enumerate}
\end{theorem}
\begin{proof}
\textbf{(1):}~Let $I_U$ be an universal arc for $f$. By
Lemma~\ref{lem:selfcover}\eqref{con:self-covering} and
Lemma~\ref{lem:universal-int-gives-nacc}, it is enough to
prove that for any $\eps>0$ there is an integer $n=n(\eps)$
such that
\[
G\setminus \interior f^{n}(I_U) \subset \ball(\Nacc(f),\eps).
\]
Suppose on the contrary, that there is $\eps>0$ such that
for every  $n\ge 0$ we can find $x_n\notin  \ball(\Nacc(f),\eps)\cup \interior f^{n}(I_U)$.
Passing to a subsequence if necessary, we can assume that $\bar{x}$ is the limit
of the sequence $\{x_n\}$. Clearly,
$\bar{x}\notin \ball(\Nacc(f),\eps)$. If $\bar{x}\in \interior f^{\bar{n}}(I_U)$
for some $\bar{n}\ge 0$, then $\bar{x}\in \interior f^{\bar{n}}(I_U)\subset \interior f^{n}(I_U)$ for
all $n$  large enough, hence $x_n\in\interior f^n(I_U)$ for some $n$,
contradicting definition of the sequence $\{x_n\}$. But then
$\bar{x}\in G\setminus\interior f^{n}(I_U)$ for all $n\ge 0$,
hence $\bar{x}\in\Nacc(f)$, which is a contradiction.

\noindent \textbf{(2):} By
Lemma~\ref{lem:mixing_graph}\eqref{con:mixing_graph_unifrom_component}, the
diameters of components of $\cl{G\setminus f^k(I_U)}=G\setminus\interior f^k(J)$
tend to $0$ as $k\to\infty$. If there were at least $\disc(G)$ points in
$\Nacc(f)$, then $G\setminus f^k(I_U)$ would have to have at least $\disc(G)$
components for $k$ large enough. But by \cite[Lemma 4.2]{Nadler}, for each $k\ge 0$
the set ${G\setminus f^k(I_U)}$ has less than $\disc(G)$ components, since
$f^k(I_U)$ is connected, a contradiction. It is clear from
Lemma~\ref{lem:selfcover}\eqref{con:universal-catching}
and \eqref{con:nacc1} above that $\Nacc(f)\neq\emptyset$ if and only if $f$ is
pure mixing.

\noindent\textbf{(3):} It is enough to show that
for every $x\in \Nacc(f)$ there is $y\in \Nacc(f)$ such that $f(y)=x$.
Assume, contrary to our claim, that there exists a point $x\in\Nacc(f)$
such that $f^{-1}(x)$ is disjoint from $\Nacc(f)$. Therefore, we can find an
$\eps>0$ such that $f^{-1}(x)\subset G\setminus B$ where $B=\cl{B(\Nacc(f),\eps)}$.
Since $x\not\in f(B)$, the set $U=G\setminus f(B)$ is an open neighborhood of $x$
such that
\[
f^{-1}(U)=f^{-1}( G\setminus f(B))\subset G\setminus B.
\]
By \eqref{con:nacc1} above, there exists $n>0$ such that
$G\setminus \cl{B(\Nacc(f),\eps)}$ is contained in $f^n(I_U)$.
Therefore $U\subset f^{n+1}(I_U)$, and $x\in\interior f^{n+1}(I_U)$,
contradicting the assumption $x\in\Nacc(f)$.

\noindent\textbf{(4):} Fix $n\ge 1$. It is clear that $\Nacc(f)\subset \Nacc(f^n)$. We will show that the converse inclusion also holds.
Let $I_U^1$ and $I_U^n$ be universal arcs for $f$ and $f^n$, respectively.
Arguing as in proof of Lemma~\ref{lem:selfcover}\eqref{con:universal-catching}
and applying \eqref{con:self-covering} of the same Lemma we get $N_U^1$ such that
$\interior I_U^1 \cup \interior f(I_U^1)\cup\ldots\cup \interior f^k(I_U^1)\subset \interior f^{N_U^1+k}(I_U^n)$
for $k\ge 0$. In particular, taking $L>0$ such that $n\cdot L > N_U^1$ we have
\[
\bigcup_{j=0}^{nl-N_U^1} \interior f^j(I_U^1)\subset \interior (f^n)^l(I_U^n)\quad\text{for every}\;l\ge L.
\]
Summing over all $l\ge L$ and taking complement of both sides finishes the proof.
\end{proof}
%
%

\section{Local behavior around inaccessible points --- inaccessible sides}

To establish the main theorem of the next section, we need to describe a local
behavior of a map $f$ around its inaccessible points. To do it rigorously
we have to introduce some technical terminology. Let $G$ be a graph with a fixed triangulation $\Lom$.
A \emph{canonical neighborhood} of a point $p\in G$ is an open set $U$ such that
$\cl{U}$ is an $n$-star, where $n=\val(p)$, and every connected component of
$U\setminus \{p\} $ is homeomorphic with $(0,1)$. Moreover, we demand canonical
neighborhoods of vertices of $G$ to have disjoint closures. If $f$ is pure mixing
map, then without lost of generality we assume that all inaccessible points are
vertices. Clearly, any point $p\in G$ has arbitrarily small canonical neighborhoods,
and there is $\eps_0>0$ such that for every $p\in G$ the open ball $\ball(p,\eps_0)$
is a canonical neighborhood of $p$. We use uniform continuity of $f$ to get an
$\eps_c>0$ such that $f(\ball(x,\eps_c))\subset\ball(f(x),\eps_0)$.
From now on we assume that with every point $p\in G$ we associated its canonical
neighborhood $U_p=\ball(p,\eps)$. A pair $(p,S_p)$, where $p\in G$ and $S_p$ is a
connected component of $U_p\setminus \{p\}$ is called a \emph{side} of $p$. By a
slightly abuse of our terminology, we will identify a side with the sole set $S_p$,
where the subscript will remind us which point in $G$ we use as a base for our side.
If $0<\delta\le \eps_0$, and $S_p$ is a side, then a \emph{ray} of length $\delta$
in the direction $S_p$ is a subset $R(S_p,\delta)=S_p\cap\ball(p,\delta)$.
Nevertheless, our analysis is local, that is, we are investigating $f$ restricted to
$\ball(\Nacc(f),\eps)$ for small $\eps>0$, our considerations will not depend of the
choices we made above. 
If $x,y$ are two points in a canonical neighborhood $U_p$ of some point $p\in G$ then
we let $\langle x,y\rangle$ to denote the convex hull of $x$ and $y$ in $U_p$.
It is well defined since $U_p$ is a tree.

\begin{description}
\item[Standing assumption] For the rest of this section we fix a graph $G$ and a
mixing map $f\colon G\defmapsto G$. We also let $I_U$ to be an universal arc  for $f$ and fix a triangulation $\Lom$ of $G$ such that all inaccessible points are vertices.
\end{description}


\begin{definition}
We say that a side $S_p$ is \emph{accessible}
if there is $n\ge 0$ such that $S_p\subset f^n(I_U)$.
A side $S_p$ is an \emph{inaccessible side} if it is not accessible.
\end{definition}

\begin{lemma}\label{lem:char-nacc-side}
A point $p\in G$ is inaccessible for $f$ (i.e., $p\in\Nacc$) if and only if it has an inaccessible side.
\end{lemma}
\begin{proof}
By Lemma~\ref{lem:universal-int-gives-nacc} and Theorem~\ref{thm:nacc}\eqref{con:nacc1},
a point $p$ is inaccessible for $f$ if and only if $p$ is not an interior
point of $f^n(I_U)$ for every $n\ge 0$. 
Equivalently, every open neighborhood of $p$ has nonempty intersection
with $G\setminus f^n(I_U)$ for every $n\ge 0$. From the above
and Lemma~\ref{lem:selfcover}\eqref{con:self-covering} we conclude
that $p\in\Nacc(f)$ if and only if there is a side $S_p$ which is
not contained in $f^n(I_U)$ for every $n\ge 0$.
\end{proof}
\begin{lemma}\label{lem:nacc-side-counting}
The map $f$ has less than $\nadler(G)$ inaccessible sides.
\end{lemma}
\begin{proof}There is $n$ such that all accessible sides and
$G\setminus\ball(\Nacc(f),\eps_c)$ are contained in $f^n(I_U)$.
It follows that every inaccessible side contains exactly one
endpoint of the subgraph $f^n(I_U)$. To finish the proof, we
conclude from \cite{Nadler} that a subgraph of $G$ has less
than $\nadler(G)$ endpoints.
\end{proof}
\begin{lemma}\label{lem:pre-img-side}
Let $p\in G$. For every side $S_p$ there is a point $q\in G$ and a side
$S_q$ such that $f(q)=p$ and $f(S_q)\cap S_p\neq \emptyset$.
\end{lemma}
\begin{proof}Let us choose an infinite sequence $\{y_n\}\subset S_p$ converging to $p$.
Since $f$ is mixing, hence surjective, there is an infinite sequence $\{x_n\}$ such that
$f(x_n)=y_n$ for all $n$. Passing to a subsequence if necessary, we may assume that
$x_n$ converges to some $q\in G$, and there is a side $S_q$ such that $\{x_n\}\subset S_q$.
By continuity $f(q)=p$, and clearly $f(S_q)\cap S_p\neq \emptyset$ as demanded.
\end{proof}

\begin{lemma}\label{lem:two-sides}
If $q\in G$ and $S_q$ is a side such that $f(S_q)$ intersect at least two sides
of $p=f(q)$, then every side $S_p$ of $p$ such that $f(S_q)\cap S_p\neq \emptyset$
is accessible.
\end{lemma}

\begin{proof}Observe that if $x,y\in S_q$ are such that $f(x)$ and $f(y)$ belongs to
different sides of $p$ then there is a path in $f(S_q)$ joining $f(x)$ with $f(y)$.
Since $f(S_q)$ is uniquely arcwise connected  this path must contain $p$. Then there
must be a point $q_0\in S_q$ such that $f(q_0)=p$. Let $S_p$ be a side such that
$f(S_q)\cap S_p\neq\emptyset$. We can choose a point $z\in S_q$ such that
$f(z)\in S_p$. Then $f(\langle z,q_0\rangle)$ contains a ray $R(S_p,\delta_0)$ for
some $\delta_0>0$. Clearly, $\langle z,q_0\rangle\subset G\setminus\ball(\Nacc(f),\delta_1)$
for some $\delta_1>0$. Let $\delta=1/2\cdot\min\{\delta_0,\delta_1\}$. By
Theorem~\ref{thm:nacc}\eqref{con:nacc1} there is an integer $N$ such that
$f^n(I_U)\supset G\setminus\ball(\Nacc(f),\delta)$ for all $n\ge N$. In particular,
$f^{N+1}(I_U)$ contains $S_p$.
\end{proof}

\begin{lemma}
\label{lem:acc-side-img}
If $q\in G$ and $S_q$ is an accessible side then every side $S_p$ of $p=f(q)$  such that
$f(S_q)\cap S_p\neq\emptyset$ is accessible.
\end{lemma}
\begin{proof}
By Lemma~\ref{lem:two-sides} it is sufficient to consider only the
case when $f(S_q)$ intersect only one side $S_p$. Then $f(S_q)$ contain
a ray in the direction $S_{p}$ and we may proceed as in the proof of Lemma~\ref{lem:two-sides}.
\end{proof}

Let $\NaccSides$ denote the set of all inaccessible sides of points in $G$.

\begin{theorem}
\label{thm:sides-function}
There is the unique bijection
$f^*\colon\NaccSides\defmapsto\NaccSides$ such that
$f^*(S_p)=S_q$ if and only if
$S_q\in\NaccSides$ is a side of $q=f(p)$
such that $f(S_p)\cap S_q\neq\emptyset$.
Moreover, for every $0<\eps<\eps_{c}$ there is $\delta>0$ such that for every
$S_p\in\NaccSides$ if $S_q=f^*(S_p)$, then $f(R(S_p,\delta))\subset R(S_q,\eps)$.
\end{theorem}
\begin{proof}
Let $S_p\in\NaccSides$. By Lemma~\ref{lem:pre-img-side} there is a point $q\in G$ and
its side $S_q$ such that $f(q)=p$ and $f(S_q)\cap S_p\neq\emptyset$. On account of
Lemma~\ref{lem:acc-side-img}, $S_q$ must be inaccessible. It follows from
Lemma~\ref{lem:char-nacc-side} that $q\in\Nacc(f)$. By the above, we may define a
function  $g^*\colon\NaccSides\defmapsto\NaccSides$ such that if $g^*(S_q)=S_p$, then
$f(S_p)\cap S_q\neq\emptyset$. By Lemma~\ref{lem:two-sides} $g^*$ must be injective,
and since $\NaccSides$ is finite, $g^*$ is a bijection. We define $f^*$ to be the inverse
of $g^*$. Now, if $f^*(S_p)=S_q$, then $f(S_p)$ is a ray in the
direction $S_q$. Moreover, $S_q$ must be unique. Now, standard
application of uniform continuity finishes the proof.
\end{proof}

\begin{corollary}\label{cor:sides-function}
\begin{enumerate}
\item If $f^*$ is a function as above, then the set $\NaccSides$ consists of
periodic orbits of $f^*$.\label{con:periodic-sides}
\item For every $0<\eps<\eps_{c}$ there is $\delta>0$ such that for every side
$S_p\in\NaccSides$  there is $1\le m <\nadler(G)$  such that $f^m(p)=p$ and
$f^m(R(S_p,\delta))\subset R(S_p,\eps)$.\label{con:periodic-rays}
\item For every $0<\eps<\eps_{c}$ there is $\delta>0$ such that for every
accessible side $S_p$ of some $p\in G$ we have
\[
f(R(S_p,\delta))\subset \ball(f(p),\eps)\setminus\bigcup_{(f(p),S)\in\NaccSides} R(S,\eps).
\] \label{con:accessible-rays}
\end{enumerate}
\end{corollary}
\begin{proof}
Both parts follows from uniform continuity of $f$ and Lemma~\ref{lem:nacc-side-counting}
and~Theorem~\ref{thm:sides-function}
\end{proof}

%

\section{Structure theorem for pure mixing graph maps}

There is a natural way to provide examples of pure mixing map of a circle:
(1) Start with an interval map such that, either both endpoints are fixed
and at least one of them is inaccessible, or both endpoints are inaccessible
and form a single cycle of length two. (2) Identify the endpoints of the
interval to obtain a circle. After the identification we still have a well
defined map, with the same number of inaccessible sides as at the beginning.
As it was noted in \cite{CovenMulvey} and elaborated in \cite{HK} all pure
mixing maps of the circle may be regarded as a result of applying this
procedure to some pure mixing interval map. We will extend this result to
a pure mixing map of an arbitrary graph.

From the previous section we see that for a mixing graph map $f$ and a point $x\in G$
either the set of all preimages of $x$ is dense in $G$, or $f^{-1}(x)=\{y\}$ where
$x,y\in\Nacc(f)$, and all sides of points lying on the (finite) orbit of $x$ are
inaccessible.
Moreover, given a pure mixing graph map $f$ as above, we may construct a new graph
$G'$ by detaching inaccessible sides from points of $\Nacc(f)$ (we keep the space
compact by adding some additional points). Since inaccessible sides are mapped onto
inaccessible sides in a one-to-one way, the map $f$ lifts in a natural way to
a new map $g\colon G'\defmapsto G'$. Any inaccessible point for $g$ is an endpoint
of $G'$ having now only finite number of inaccessible preimages.

The proof is only a formalization of the procedure described above.

\begin{theorem}[Structure Theorem]\label{thm:structure}
Let $f\colon G\defmapsto G$ be a pure mixing graph map. Then there exist a graph $G'$,
a pure mixing map $g\colon G'\defmapsto G'$, and a continuous surjection
$\pi\colon G'\defmapsto G$ such that:
\begin{enumerate}
\item The map $f$ is factor of $g$ via $\pi$, that is $\pi \circ g = f \circ \pi$.
Moreover, $\pi$ is one-to-one on $G'\setminus \Nacc(g)$ and $\pi(\Nacc(g))=\Nacc(f)$.
    \label{thm:structure:c1}
\item If $e\in \Nacc(g)$ then $e$ is an endpoint of $G'$ and $g^{-1}(e)=\{e'\}$ for
some $e'\in \Nacc(f)$. Moreover, $\Nacc(g)$ has less then $\nadler(G)$ elements.
    \label{thm:structure:c2}
\end{enumerate}
\end{theorem}
\begin{proof}
Let $\mathcal{L}$ be a triangulation of $G$. By the definition of canonical neighborhood each side $S_p\in\NaccSides$
is contained in exactly one edge of $\mathcal{L}$. Let $X$ denote the disjoint union of
edges of $\mathcal{L}$, so $X$ is homeomorphic with $[0,1]\times\{1,\ldots,l\}$, for
some $l>0$. Any edge of $G$ can be now identified with a component of $X$.

There is a unique equivalence relation $R$ on $X$ such that $G$ is the identification
space (quotient space) $G=X/R$. The relation $R$ is called the incidence relation and
informs us which endpoints are to be attached to form $G$.

Obviously, we can have $xRy$ for $x\neq y$ only if $x,y$ are endpoints of some
components of $X$. We can view equivalence classes with respect to $R$, denoted by
$[x]_R$ as elements of $G$. Moreover, if the class $[x]_R$ has more than one element
then it represents a vertex $v$ of $G$ and consists of $\val (v)$ points of $G$. Let
$\NaccSides$ denote the set of all inaccessible sides of points in $G$.
Then sides from $\NaccSides$ are in a one-to-one correspondence with a subset
of endpoints of components of $X$. Therefore, we may write $\NaccSides\subset X$ by
convenient abuse of notation. We define $\mathcal{R}=X\setminus\NaccSides$, and call
a point $x\in X$ regular if and only if $x\in\mathcal{R}$.

We define a new relation $R'\subset R$ by declaring $xR'y$ if and only if $xRy$ and
either, both $x$ and $y$ are regular, or $x=y$. Clearly, $R'$ is an equivalence
relation, and $[x]_{R'}\subset [x]_R$ for $x\in X$. Moreover, the space $G'=X/R'$ is
easily seen to be homeomorphic with the space obtained by removing from each
inaccessible side for $f$ on $G$ a tiny ray lying on that side. Hence, $G'$ is a graph.

Now, we will define a map $g\colon G' \defmapsto G'$. First, note that if $x$ is a
regular point such that $[x]_R\notin\Nacc(f)$ then $[x]_R=[x]_{R'}$.
And if we denote  $[y]_R=f([x]_R)$ then also $[y]_R=[y]_{R'}$. This is true, because
every side of $[y]_R$ is accessible. In that case, we define $g([x]_{R'})=f([x]_R)=[y]_{R'}$.

If $x=S_v\in\NaccSides$, that is, if $x$ is an endpoint representing some inaccessible
side of a point $v=[x]_R\in G$, then there is a point $y\in X$ representing a side
$f^*(S_v)$, where $f^*$ is a map defined in Theorem~\ref{thm:sides-function}. In particular,
$y\in\NaccSides$. Then we have $[x]_{R'}=\{x\}$, $[y]_{R'}=\{y\}$ and we may
define $g([x]_{R'})=[y]_{R'}$.

It remains to consider the case when $x$ is a regular point, and $[x]_R\in f^{-1}(\Nacc(f))$.
In that case all sides of $[x]_R$ are accessible, hence $v=f([x]_R)$ must have accessible and
inaccessible sides. Therefore $v=[y]_R$ for some regular point $y$. Clearly, if $y'$ is
another regular point such that $v=[y']_R$ then $yR'y'$ by the definition of $R'$. It follows that
we may define $g([x]_{R'})=[y]_{R'}$.

Now continuity easily follows from Theorem~\ref{thm:sides-function} and Corollary~\ref{cor:sides-function}.
Other points are also easy to see.

%
%

\end{proof}



\section{Transitivity and entropy of pure mixing graph maps}

With the structure theorem at hand
we can now study the topological entropy of pure mixing graph maps.
In this (and next sections)
we will utilize the structure theorem and other tools.


To estimate the topological entropy of pure mixing graph maps
we need the notion of a loose horseshoe from \cite{HK}.
Recall that an $s$-\emph{horseshoe} for $f$ is a free
arc $J$ contained in
the domain of $f$, and a collection
$\mathcal{C}=\{A_1,\ldots,A_s\}$ of $s\ge 2$ nonempty compact
subsets of $J$ fulfilling the following three conditions: (a)~each
set $A\in \mathcal{C}$ is an union of finite number of arcs,
(b)~the interiors of the sets from $\mathcal{C}$ are
pairwise disjoint, (c)~$J\subset f(A)$ for every $A\in\mathcal{C}$.
If the union of elements of $\mathcal{C}$
is a proper subset of $J$, or $J$ is a~proper subset of $f(A)$ for
some $A\in\mathcal{C}$ then we say that a horseshoe
$(J,\mathcal{C})$ is \emph{loose}.
The following lemma is adapted from \cite{HK} and summarizes
results of  \cite[Section 4.2]{HK}. It is easy to see
that the assumption that the graph is an interval $[0,1]$ or
a circle was inessential there, and the result holds for
arbitrary graph.

\begin{lemma}[\cite{HK}] \label{prop:strong_horseshoe_gives_entropy}
If a transitive graph map $f$ has a loose $s$-horseshoe then \mbox{$h(f)>\log s$}.
\end{lemma}

The next two theorems provide a lower bounds of topological entropy of pure mixing
graph map. The first of these facts comes from
\cite[Proposition 4.2]{ABLM} for the tree maps and with the weak inequality.
Later Baldwin in \cite{Baldwin} observed that the inequality is in fact strict.
Here we present a variant for that result which is valid for graph maps.

\begin{theorem}\label{thm:3-horseshoe}
Let $f$ be a transitive map of a graph $G$. If $e$ is an endpoint of
$G$ such that $f^{-1}(e)=\{e\}$, then $e$ is an accumulation point of
fixed points of $f$ and $h(f)> \log 3$.
\end{theorem}
\begin{proof}We identify the edge containing $e$ with the unit
interval $[0,1]$ with $e=0$, and we use the induced ordering $<$
and write about intervals, etc. In addition, we agree to write
$x<y$ for any $x\in [0,1]$ and $y\in G\setminus [0,1]$. We can
find an $0<\eps_0<1$ such that $f([0,\eps_0])\subset [0,1]$.
Observe that for any $0<\eps<\eps_0$ we have:
\begin{description}
\item[($\star$)]\label{con:by_trans1}
If $\eps'=\min f(G\setminus[0,\eps))$, then $0<\eps'<\eps$
(since $\eps'\ge \eps$ would imply that $G\setminus[0,\eps)$
is a proper invariant set with nonempty interior contradicting
transitivity, and $\eps'=0$ would contradict $f^{-1}(e)=\{e\}$).
\item[($\star\star$)]\label{con:by_trans2}
If $\eps''=\max f([0,\eps])$, then $\eps<\eps''$
(as $\eps''\le \eps$ would imply that $[0,\eps]$ is invariant
contradicting transitivity).
\end{description}
We will use ($\star$) and ($\star\star$) all the time without
any reference. To prove that $e$ is an accumulation point of
fixed points of $f$ it is enough to show that for any
$0<\eps<\min f(G\setminus [0,\eps_0))$ we have $f(a)>a$ and
$f(b)<b$ for some $0<a<b<\eps$. To see it let
$\eps'=\min f(G\setminus[0,\eps))$.  Then there is
$b\in [\eps',\eps)$ such that $f(b)<b$. There also must be
$0<a<b$ such that $f(a)=\max f([0,b])>b$.

For the proof that $h(f)> \log 3$,
let $0<a<\min f(G\setminus [0,\eps_0))$ be a fixed point of $f$,
and set $b=\min f(G\setminus [0,a))$. We have $0<b<a$, and since
$e$ is an accumulation point of fixed points of $f$ there are
fixed points $s$ and $t$ of $f$ such that $0<s<b$, $s<t$ and
$\max f([0,b])=f(z)$ for some $s<z\le b$. Moreover, we may assume
that there is no fixed point in $(s,t)$, so $f(x)>x$ for all
$x\in(s,t)$. It follows that $t<a$ and there must be fixed points
$u$ and $v$ of $f$ such that $t\le u<v\le a$ and
$\min f(G\setminus[0,s))=f(w)$ for some $u<w<v$. Again we may assume
that $f(y)<y$ for all $y\in(u,v)$. There
are fixed points $p$ and $q$ such that for some $p<r<q$ we have
$f(r)=\max f([0,v])=\max f([p,q])$. Clearly, we have $s\le p< q\le u$,
since $\max f([0,s])\leq \max f([s,t])$ and $\max f([0,v])\leq \max f([0,u])$.
Therefore, $[p,r]$, $[r,w]$,
$[w,v]$ form a loose $3$-horseshoe for $f$, as $f(r)>v$, and $f(w)<s$.
\end{proof}

\begin{theorem}\label{thm:3-horseshoe-kappa}
If $f$ is a pure mixing map of a graph $G$, then
\begin{enumerate}
\item $h(f)>(1/\nadler(G))\cdot\log 3$,
\item there exists $0<m<\nadler (G)$ such that $f^m$ has
infinitely many fixed points.
\end{enumerate}
\end{theorem}
\begin{proof}By Theorem~\ref{thm:structure} $f$ is a factor of a map
$g\colon G'\defmapsto G'$ such that $g^{-1}(\Nacc(g))=\Nacc(g)$. Moreover,
$\Nacc(g)$ has less than $\nadler(G)$ elements, so there is
$0<m<\nadler (G)$ such that $g^{-m}(e)=\{e\}$ for some endpoint of
$G'$. By Theorem~\ref{thm:3-horseshoe} we see that $h(g^m)>\log 3$. As
$g^m$ is an extension of $f^m$ via finite-to-one semiconjugacy, we have
$h(f^m)=h(g^m)>\log 3$. Moreover, $g^m$ has infinitely many fixed points,
and so does $f^m$.
\end{proof}

As a direct consequence we get the following lower bound for $\inf(h(\pure{G}))$
for a given graph $G$.

\begin{corollary}\label{cor:inf-pure}
If $G$ is a graph, then $\inf(h(\pure{G}))\ge (1/\nadler(G))\cdot\log 3$.
\end{corollary}

As the second corollary we get the following theorem which is contained in
\cite{Yokoi} as Theorem 4.1. Here, $\operatorname{Fix}(f^k)$ denotes the
set of all fixed points of $f^k$.

\begin{theorem}\label{thm:Yokoi-main}
Let $f\colon G \defmapsto G$ be a graph map with
$\#\operatorname{Fix}(f^k)<\infty$ for each $k\geq 1$. If $f$ is
transitive, then it is strongly transitive, that is, for every
non-empty open set $U\subset G$ there exists $n>0$ such that
\[G=U\cup f(U) \cup\ldots \cup f^n(U).\]
\end{theorem}

\begin{proof}
If $f$ is an irrational rotation then the result is well-known.
Assume that $f$ is non-invertible. By Theorem \ref{thm:trans-not-tot-trans}
it is enough to prove that if $f$ is totally transitive then it must be
exact. But it is well known (see \cite{HKO} for a simple proof of that fact)
that every totally transitive, and non-invertible
graph map must be mixing. By Theorem \ref{thm:3-horseshoe-kappa} a mixing graph map with
$\#\operatorname{Fix}(f^k)<\infty$ for each $k\geq 1$ must be exact.
\end{proof}

We recall that a map $f\colon X\defmapsto Y$ is monotone if $X$ and $Y$ are
topological spaces, $f$ is continuous, and for each point $y\in Y$ its
preimage $f^{-1}(y)$ is connected. If
$X$ is a tree and there is a finite set $P\subset X$ such that for each \
connected component $C$ of $X\setminus P$ the map $f|_{\cl{C}}\colon \cl{C}\defmapsto Y$
is monotone, then we say that $f$ is $P$-\emph{monotone}.
We say that a tree map $f$ \emph{piecewise monotone} if $f$ is $P$-monotone with
respect to some finite $P\subset T$. If $f$ is a $P$-monotone map then the
closures of connected components of $T\setminus P$ are called
$P$-\emph{basic intervals} for $f$.

Theorem~\ref{thm:Yokoi-main} is a generalization of the following well-known result (see \cite{Yokoi} for more comments).

\begin{corollary}\label{cor:Yokoi-cor}
Let $f\colon G \defmapsto G$ be a piecewise monotone graph map. If $f$
is transitive then it is strongly transitive, in particular
a totally transitive piecewise monotone graph map must be exact.
\end{corollary}
\begin{proof}
It is easy to see that every transitive and piecewise monotone map fulfills the assumption of Theorem
\ref{thm:Yokoi-main} (by transitivity, images of nondegenerate continua remain nondegenerate).
\end{proof}

\section{Topological entropy of transitive tree maps}

In this section we collect some technical results which we will use
in the next section. We will also need some additional terminology,
which we recall below.

Let $f\colon T\defmapsto T$ be a tree map. We say that $f$ is \emph{linear} on a
set $S\subset T$ if there is a constant $\alpha$ such that
$d(f(x),f(y))=\alpha d(x,y)$ for all $x,y\in S$ (here, as always $d$
denotes the taxicab metric on $T$). If a $P$-monotone map $f$ is linear on
every $P$-basic interval then we call it $P$-\emph{linear} or
\emph{piecewise linear} if there is no need to single out $P$. A
$P$-monotone map $f$ is \emph{Markov map} if $P\subset T$ contains all
vertices of $T$ and $f(P)\subset P$. In the above situation we call $f$ a
$P$-\emph{Markov} map for short.  If $f$ is a $P$-Markov map, then the
\emph{Markov graph} of $f$ with respect to $P$ ($P$-Markov graph of $f$ for
short) is then defined as a directed graph with the set of $P$-basic
intervals as a set of vertices and with the set of edges defined by the
$f$-covering relation, that is, there is an edge from a $P$-basic interval
$I$ to $P$-basic interval $J$ in the Markov graph ($I\to J$) if and only if $J$ is
$f$-covered through $f$ by $I$. A \emph{path (of length $n$)} in a graph
$\mathcal{G}$ is any sequence $I_0,I_1,\ldots,I_n$ of vertices of
$\mathcal{G}$ such that there exists an edge $I_{i-1}\to I_{i}$ for each $i=1,\ldots,n$.
A \emph{cycle} of length $n$ is any path $I_0,I_1,\ldots,I_n$ of length
$n$ such that $I_0=I_n$. Graph is \emph{strongly connected} if for any
pair of its vertices $I,J$ there is a path $I_0,I_1,\ldots,I_n$ in
$\mathcal{G}$ with $I=I_0$ and $J=I_n$. If $I_1,\ldots,I_n$
is an enumeration of the set of $P$-basic intervals then the
\emph{incidence matrix} of $f$ with respect to $P$ is a $n$ by $n$ matrix
$A=[a_{ij}]$ with $a_{ij}=1$ if $I_i\to I_j$ and $a_{ij}=0$ otherwise. The
\emph{spectral radius} of a  square complex matrix is defined as the largest
absolute value of its eigenvalues.


There are also few results we would like to recall for later reference.
The following Lemma comes from \cite[Corollorary 1.11]{Baldwin}.

\begin{lemma}\label{lem:Markov-transitive}
Suppose $T$ is a tree and $f\colon T\defmapsto T$ is $P$-Markov and $P$-linear with
respect to some $P$ containing all nodes of $T$. Then $f$ is transitive if and only
if the $P$-Markov graph of $f$ is strongly connected and is not a graph of a cyclic
permutation.
\end{lemma}

The following lemma is well-known. We restate it in a suitable form.

\begin{lemma}\label{lem:Markov-entropy}
Let $f\colon T\defmapsto T$ be a Markov tree map. If $\mathcal{G}$ is any Markov
graph of $f$ and $A=[a_{ij}]_{n\times }$ is the corresponding incidence matrix
with a spectral radius $\rho\ge 0$, then $h(f)=\log \rho$ if $\rho>0$, and
$h(f)=0$ otherwise.
Moreover, if there is $s>0$ such that for every vertex $v$ of $\mathcal{G}$
the number of directed paths of length $n>0$ starting at $v$ in $\mathcal{G}$
is bounded from the above by $s^n$, then $h(f)\le \log s$.

If in addition $\mathcal{G}$
is strongly connected and there are $t>0$ and $n>0$ such that for some vertex
$v$ of $\mathcal{G}$ the number of paths of length $n>0$ starting at $v$ in
$\mathcal{G}$ is bounded below by $t^n$, then $ \log t\le h(f)$.\end{lemma}
\begin{proof}The connection between the spectral radius of the incidence
matrix and topological entropy is well known (see \cite[Proposition 1.4]{Baldwin}).
It is easy to see that the number of paths of length $n>0$ starting at $v$ in
$\mathcal{G}$ is equal to the row-sum of row $v$ in $A^n$. Moreover, the incidence
matrix of a graph is irreductibile if and only if the graph is strongly connected.
The bounds on the spectral radius $\rho$ of irreducible square matrix
$B=[b_{ij}]_{i,j=1}^m$ comes from well known formula (e.g. see \cite[Exercise 4.2.3., p.111]{ISDC}):
\[
\min_{i=1,\ldots,m}\sum_{j=1}^m b_{ij} \le \rho \le \max_{i=1,\ldots,m}\sum_{j=1}^m b_{ij}.
\]

\end{proof}


The following theorem is a restatement of \cite[Theorem 3.2.]{Baldwin}.
We added the assertion $(\star)$, which is not included in the original
statement, but is a consequence of the proof provided in \cite{Baldwin}.

\begin{theorem}\label{thm:edge-adding}
Let $T$ be a tree, and $f\colon T\defmapsto T$ be transitive and Markov. Let $z$ be a
fixed point of $f$, and let $T'$ be the tree which is obtained from $T$ by attaching
an arc to $z$ at one of the endpoints of the arc. Then for every $\eps>0$, there is
a transitive Markov map $f'\colon T'\defmapsto T'$ such that $h(f')<h(f)+\eps$ and both
ends of the new arc are fixed for $f'$.
Furthermore, $f'$ can be defined so that, in addition,
\begin{description}
\item[$(\star)$] If some subset of the endpoints of $T$ forms an $f$-cycle $C$, then
$C$ is also a periodic orbit of $f'$.
\end{description}
\end{theorem}



Let $(X,x_0)$ and $(Y,y_0)$ be topological pointed spaces (spaces with distinguished
basepoints). The wedge sum of $(X,x_0)$ and $(Y,y_0)$ (denoted as
$(X,x_0)\wedge(Y,y_0)$ is the quotient of the disjoint union of $X$ and $Y$ by the
identification $x_0 \sim y_0$. The $m$-th wedge power $(X,x_0)^{\wedge m}$ is defined
as wedge sum of $m$ copies of $(X,x_0)$.

\begin{lemma}\label{lem:wedge-entropy}
Let $X$ be a topological space and $f\colon X\defmapsto X$ be a transitive map.
If $x_0$ is a fixed point of $f$, then for every $m\geq 2$ there is a transitive,
but not totally transitive map
$F\colon (X,x_0)^{\wedge m}\defmapsto (X,x_0)^{\wedge m}$ such that $h(F)=h(f)/m$.
Moreover, $x_0$ is the unique fixed point of $F$, and all other periodic orbits of $f$ are
periodic points of $F$ with $m$ times longer primary periods (formed by the $m$-times copy
of points of the orbit of $f$). Furthermore, if $X$ is a tree and $f$ is a Markov map, then the same hold for
$(X,x_0)^{\wedge m}$ and $F$ respectively.
\end{lemma}
\begin{proof}Let us identify a disjoint union of $m$ copies of $X$ with
$X\times\{0,\ldots,m-1\}$. Define a map
$F'\colon X\times\{0,\ldots,m-1\}\defmapsto X\times\{0,\ldots,m-1\}$ by
\[
F'((x,i))=
\begin{cases}
(x,i+1),&\text{for}\; i=0,\ldots,m-2,\\
(f(x),0),&\text{for}\;i=m-1.
\end{cases}
\]
It is easy to see that $F'$ induces a quotient map $F$ on
$(X,x_0)^{\wedge m}$ with the desired properties.
The rest of the proof is now straightforward.
\end{proof}

Note, that if $m\ge 2$, then the map constructed in Lemma \ref{lem:wedge-entropy} is
not totally transitive. To remove this problem we need the following Lemma which
gives a general method of constructing
exact Markov maps from transitive but not totally transitive Markov maps
and enables us to control the entropy.
The proof of the upper bound for entropy follows the ideas of \cite{Baldwin}.

\begin{lemma}\label{lem:tot-trans-entropy}
If $f$ is a transitive $P$-Markov and $P$-linear map of a tree $T$, which is not
totally transitive, then for every $\eps>0$, there is a totally transitive (hence, exact)
$P'$-Markov and $P'$-linear  map $f'\colon T\defmapsto T$ such that $h(f')<h(f)+\eps$.
Furthermore, $f'$ can be defined so that, if some subset of the endpoints of $T$
form an $f$-cycle $C$, then $C$ is also a periodic orbit of $f'$.
\end{lemma}
\begin{proof}
Let $f\colon T\defmapsto T$ be a transitive, but not totally transitive
$P$-Markov and $P$-linear tree map. We denote the $P$-Markov graph of $f$ by
$\mathcal{G}$. Fix $\eps>0$ and let $r,s>1$ be such that $h(f)<\log r < \log s < h(f)+\eps$.
By Theorem \ref{thm:trans-not-tot-trans} $f$ has the unique fixed point $p\in T$ and
we can enumerate the closures of connected components
of $T\setminus\{p\}$ by $T_0,T_1,\ldots,T_{n-1}$ with $n\ge 1$ in such a way that
$f(T_{j})=T_{j+1}$ for $j=0,\ldots,n-1$, and $f(T_{n-1})=T_{0}$ hold. Without loss of
generality we may assume that $p\in P$. Moreover,
for each $j=0,1,\ldots,n-1$ the map $f^{n}|_{T_j}\colon T_j\defmapsto T_j$ is a
totally transitive (thus, exact) piecewise linear Markov map, in particular, there
is a point $q\neq p$ such that $f(q)=p$. Without loss of generality we may
assume that $q\in T_0$. Let $[x,y]$ be a $P$-basic interval in $T_0$ containing
$q$. If $q\notin P$, then either $f|_{[x,y]}$ would be not monotone, or $f(T_0)$ would
intersect  $T_j$ with $j\neq 1$. Hence $q\in P$, and without loss of generality
we may that assume $x=q$. For $j=0,1,\ldots,n-1$ let $I_j=[p,z_j]$ denote the $P$ basic
interval in $T_j$ containing $p$. 

Let us choose $\Lambda>0$ such that $3r^L<s^L$, for all $L\ge \Lambda$, where $r$ and
$s$ are as above. Then there is a point $w_0$ in the interior of $[p,z_0]$ and
$L\ge \Lambda$ such that for $k=0,\ldots, L-1$ and $j=(k \mod n)$
we have $w_k=f^{k}(w_0)\in \interior I_j$, and $f^{L}(w_0)\in P$.

Let $P''=P\cup \{w_j:j=0,\ldots,L-1\}$, and let $\mathcal{G}''$ be the Markov graph
of $f$ with respect to $P''$.
We define $f'\colon T\defmapsto T$ by modifying $f$ on $[x,y]$ only, that is,
we put $f'(z)=f(z)$ for all $z\in T\setminus [x,y]$. Next, we choose two points
$x'$, $y'$ in $[x,y]$ with $x<y'<x'<y$, then we set $f'(y')=w_0$ and $f'(x')=p$, and
extend $f'$ to $[x,x']$ by making it linear on $[x,y']$ and
$[y',x']$. We identify $[x',y]$ with $[x,y]$ by a linear homeomorphism $\psi$ such
that $\psi(x')=x$ and for $z\in[x',y]$ we define $f'(z)=f(\psi(z))$. Then $f'$ is a
continuous map of $T$, which is $P'$-Markov and $P'$-linear with respect to
$P'=P''\cup\{x',y'\}$.
Note that if some subset of the
endpoints of $T$ form an $f$-cycle $C$, then $C$ is also a periodic orbit of $f'$
since $f'|_P=f|_P$.

We claim that $f'$ is totally transitive. First, observe that $f'(T_0)$ contains
$f'([x,y])$, and hence it intersects interiors of both, $T_0$ and $T_1$.
As $p$ is the unique fixed point of $f'$, and $T_0,\ldots,T_{n-1}$ are no longer
invariant for $(f')^n$ it follows from Theorem \ref{thm:trans-not-tot-trans} that
$f'$ is totally transitive provided $f'$ is transitive.

To show that $f'$ is transitive we consider the $P'$-Markov graph of $f'$,
denoted by $\mathcal{G}'$. By Lemma~\ref{lem:Markov-transitive} it is enough
to prove that $\mathcal{G}'$ is strongly connected, since it is clear that it is
not a graph of a cyclic permutation. We will say that a $P'$-basic
interval $J$ of $f'$ is an ``old'' one if $J$ is, either a $P''$-basic interval of $f$,
or $J=[x',y]$. With this nomenclature $\mathcal{G}''$ has two ``new'' vertices, that is
$P''$-basic intervals $J'=[x,y']$ and $J''=[y',x']$. Observe that
the subgraph of $\mathcal{G}'$ given by the set of old $P'$ basic intervals together with
all edges between them is isomorphic to $\mathcal{G}$. Moreover, if any old $P'$-basic
interval $f$-covered $[x,y]$, then it also $f'$-covers $[x',y]$, $J'$, and $J''$, and the last
two intervals $f'$-cover an old $P'$ basic interval $[w_0,p]$. It easily follows that
$\mathcal{G}'$ is strongly connected and thus $f'$ is transitive. The proof of the claim is completed.

To estimate the topological entropy of $f'$ we fix a  vertex $v$ of $\mathcal{G}'$ and
provide a bound on the number of paths of length $L$ in $\mathcal{G}'$ starting
at $v$.
By the definition of $f'$ and our choice of $L$ every path of length $L$ in $\mathcal{G}'$
can pass at most once through $J'$ or $J''$. Moreover, to every path $\alpha$ of length $L$
in $\mathcal{G}$ corresponds, either exactly one path in $\mathcal{G}'$ if $\alpha$
does not contain $[x,y]$, or exactly three paths in $\mathcal{G}'$ otherwise.
By Lemma \ref{lem:Markov-entropy}, there are at most $r^L$ paths of length $L$ starting at any fixed
vertex of $G$, and we conclude that there are at most $3r^L$ paths of length $L$ in
$\mathcal{G}'$ starting at $v$. Using Lemma \ref{lem:Markov-entropy} and by our choice of $s$ we have
that $h((f')^L) \le \log (3r^L) < \log s^L$, hence $h(f')< \log s <h(f)+\eps$.
\end{proof}

The last lemma shows how to construct pure mixing examples from exact Markov tree maps
with the topological entropy as small as possible.

\begin{lemma}\label{lem:pure-mix-entropy}
Let $f\colon T\defmapsto T$ be an exact Markov tree map. If $\mathcal{O}$ is a single
periodic orbit for $f$ with $m>0$ elements contained in the set of endpoints of $T$
then for every $\eps>0$, there is a pure mixing map $f'\colon T\defmapsto T$ such that
\[
h(f') < \max\{h(f),\log 3/m\}+\eps\]
and $\mathcal{O}=\Nacc(f')$ is a periodic orbit of $f'$.
\end{lemma}
\begin{proof} Fix any $\eps>0$. Let $\mathcal{O}$ be a single periodic orbit of $f$
with $m>0$ elements contained in the set of endpoints of $T$. Set
$\eta=\max\{h(f),\log3/m\}$.
Choose $r,s>1$ such that $\eta<\log r<\log s<\eta+\eps$,
and fix $L>0$ such that $3r^L < s^L$.
Take any $P\subset T$ such that $f$ is Markov
with respect to $P$. Choose any $o_0\in\mathcal{O}$ and for $j=1,2,\ldots,m-1$ put
$o_j=f^j(o_0)$, and let $I_j=[o_j,z_j]$ denote the $P$ basic interval containing the
endpoint $o_j$.  Reasoning as in the proof of
Lemma~\ref{lem:tot-trans-entropy} we can find a $P$-basic interval $[x,y]$ such that
$f(x)=o_0$ and $x\neq o_{m-1}$. Since $f$ is exact, it cannot collapse any of intervals $I_j$, and
thus $[x,y]\neq I_j$ for any $j$. We order
$[x,y]$ such that $x<y$ and we choose two points $x'$, $y'$ in $[x,y]$ with
$x<y'<x'<y$.  We extend the tree $T$ by attaching an interval $[a_j^\infty,o_j]$ at each
$o_j\in \mathcal{O}$ such that for each $j=0,\ldots,m-1$ there is an isometry
$\varphi_j\colon [a_j^\infty,o_j]\defmapsto [a_{j+1}^\infty,o_{j+1}]$ (here we agree that
$o_m=o_0$, and $a_m^\infty=a_0^\infty$). Let us order each interval $[a_j^\infty,o_j]$
such that $a_j^\infty<o_j$. We obtained a new tree $T'$, which is homeomorphic with the
original $T$, and $T\subset T'$. Now, to finish the proof it is enough to construct
$f'\colon T'\defmapsto T'$ such that the entropy bounds hold and
$\mathcal{O}'=\{a_0^\infty,\ldots,a_{m-1}^\infty\}$ is a non-accessible periodic orbit
of $f'$. First set $f'(x)=f(x)$ for each $x\in T\setminus [x,y]\subset T'$. Observe
that there is $u_0\in I_0$ such that for $j=1,\ldots,mL$ and $l=j\mod m$ we have
\[
f'([o_j,(f')^j(u_0)])\subset I_l,\qquad\text{and}\qquad (f')^{mL}(u_0)=z_0.
\]
Next, we choose a sequence $\{a_j\}_{j=0}^\infty\subset[a_{m-1}^\infty,o_{m-1}]$
such that for each $j=1,2,\ldots$ we have
\[
a_j<a_{j-1}
\qquad\text{and}\qquad
a_0=o_{m-1}
\qquad\text{and}\qquad
\lim_{j\to+\infty}a_j=a_{m-1}^\infty.
\]
Set $\alpha_j=\varphi_{m-1}(a_j)$ for $j=0,1,\ldots$, and find infinite sequences
$\{b_j\}_{j=1}^\infty$, $\{c_j\}_{j=1}^\infty$, $\{d_j\}_{j=1}^\infty$, and
$\{e_j\}_{j=1}^\infty$ such that for each $j=1,2,\ldots$ we have
\[
a_j< c_{2j-1} < b_j < c_{2j} < e_{2j} < d_j < e_{2j-1} < a_{j-1}.
\]
For $k=0,\ldots,m-2$ and $x\in [a_k^\infty,o_k]$ set $f'(x)=\varphi_k(x)$.
For $j=\infty$ put $f'(a_{j})=\alpha_{j}$ and for each $j=1,2,\ldots$ set
\[
f'(a_{j})=\alpha_{j},\qquad f'(c_{2j-1})=f'(c_{2j})=\alpha_{j-1},\qquad
f'(e_{2j-1})=f'(e_{2j})=\alpha_j,
\]
and extend $f'$ linearly to $[a_{j},c_{2j-1}]$, $[c_{2j},e_{2j}]$, and
$[e_{2j-1},a_{j-1}]$. Now for each $j=1,2,\ldots$ there are points
$u_{j}\in [\alpha_j,\alpha_{j-1}]$, and $w_{j-1}\in[\alpha_{j},\alpha_{j-1}]$ such that
the following conditions hold:
\begin{description}
\item[$(\star)$] for each $k=1,\ldots, L$ we have
\[
(f')^{mk-1}([\alpha_{j},u_{j}])\subset [a_j,c_{2j-1}],
\qquad\text{and}\qquad
(f')^{mk-1}([w_{j-1},\alpha_{j-1}])\subset [e_{2j-1},a_{j-1}],
\]
hence $(f')^i$ is well defined on intervals $[\alpha_j,u_j]$, $[w_{j-1},\alpha_{j-1}]$ for
$i=1,\ldots,mL-1$
\item[$(\star\star)$] we have
\[
(f')^{mk-1}(u_{j})=c_{2j-1},
\qquad \text{and}\qquad
(f')^{mk-1}(w_{j-1})=e_{2j-1}.
\]
\end{description}
We have defined two sequences $\{u_j\}_{j=0}^\infty$ and $\{w_j\}_{j=0}^\infty$.
For $j=1,2,\ldots$ we define
\[
f'(b_j)=u_{j-1}\qquad \text{and}\qquad f'(d_j)=w_j,
\]
and extend $f'$ linearly to $[c_{2j-1},b_j]$, $[b_j.c_{2j}]$, $[e_{2j},d_j]$, and
$[d_j,e_{2j-1}]$. Finally, we set
$f'(x')=o_0$, $f'(y')=w_0$, and extend $f'$ linearly to $[x,y']$ and $[y',x']$.
We identify $[x',y]$ with $[x,y]$ by a linear homeomorphism $\psi$ such that
$\psi(x')=x$ and for $z\in[x',y]$ we define $f'(z)=f(\psi(z))$. Then $f'$ is a
continuous map of $T'$. Reasoning as in \cite{HK} one gets that $f'$ is a pure
mixing map, and it is clear that $\mathcal{O}'=\{a_0^\infty,\ldots,a_{m-1}^\infty\}$
is a non-accessible periodic orbit of $f'$.


To estimate the topological entropy of $f'$, for each $j=1,2,\ldots$ we define
a sequence of maps $\{f'_j\}_{j=1}^\infty$ such that
$h(f'_j)\to h(f')$ as $j\to\infty$, and for each $j=1,2,\ldots$ we
have $h(f'_j) \le \eta+\eps$.
To this end set $Q_0^j=[\alpha_{j},\alpha_{j-1}]$,
and $Q_k^j=\varphi_{k-1}(Q_{k-1}^j)$.
Let
\[
\Omega_k^j=\bigcup_{l=1}^j Q_k^l
\qquad\text{and}\qquad
S_j=\bigcup_{l=1}^j\bigcup_{k=0}^{m-1}Q_k^l.
\]

Next, we define a sequence 
of linear Markov maps $f'_j\colon T'\defmapsto T'$ by
\[
f'_j(x)=\begin{cases}f'(x),&\text{for }x\in T\cup S_j,\\
\varphi_k(x),& \text{if } x\in [a_k^\infty,o_k]\setminus \Omega_k^j \text{ for some }k=0,\ldots,m-1.
\end{cases}
\]


Let $X_j=\{x\in T' : (f')^n(x)\in T\cup S_{j} \text{ for }n=0,1,\ldots\}$.
Observe that $f'|_{X_j}=f'_j|_{X_j}$ and $h(f'_j)=h(f'_j|_{X_j})$.
As $h(f'|_{X_j})\le h(f')$ and $X_j\subset X_{j+1}$ we also have
$h(f'_{j})\le h(f'_{j+1})$, and we get
\[
\limsup_{j\to\infty} h(f'_j) = \sup h(f'|_{X_j}) \le h(f').
\]
Since $f'_j$ converges uniformly to $f'$ on $T'$, and the entropy function
is lower semicontinuous, we get
\[
h(f')\le \liminf_{j\to\infty} h(f_j).
\]
To finish the proof
we need to show that $h(f'_j) \le \eta+\eps$ for every
$j=1,2,\ldots$. But by the way $f'_j$ is defined it is a linear Markov map on $T\cup S_j$, so
Lemma \ref{lem:Markov-entropy} applies,  and the upper bound can be obtained
by counting paths in the Markov graph of $f'_j$ in a similar way as in the proof of Theorem~\ref{lem:tot-trans-entropy}.
The details are left to the reader.
\end{proof}

\section{Examples}

In this section we construct a few examples of pure mixing graph maps
which will prove that the lower bounds for $\inf(h(\pure{G}))$ derived from Corollary \ref{cor:inf-pure}
are in some cases equal to the infimum, hence we solve our main Problem
in these cases.

Recall that an $n$-star is a tree $T_n=([0,1],0)^{\wedge n}$, where $n\ge 1$.

\begin{lemma}\label{lem:n-star-entropy}Let $n>1$ and let
$T_n$ be a star with $n$ endpoints. Then for
every $\eps>0$, there is
\begin{enumerate}
\item an exact Markov map $F_\eps\colon T_n\defmapsto T_n$ such that
$(\log 3)/n\le h(F_\eps)<(\log 3)/n+\eps$.
\item a pure mixing map $G_\eps\colon T_n\defmapsto T_n$ such that
$(\log 3)/n\le h(G_\eps)<(\log 3)/n+\eps$.
\end{enumerate}
Moreover, all endpoints of $T_n$ form a single periodic orbit of $F_\eps$ and $G_\eps$.
\end{lemma}
\begin{proof}First observe that
for each $n\ge 2$ there is a transitive Markov map $f_n$ of an $n$-star $T_n$ such that
all endpoints of $T_n$ form a single cycle for $f_n$ and $h(f_n)=\log 3/n$.
It is a consequence of Lemma \ref{lem:wedge-entropy} applied to the
$3$-horseshoe map given by $f(x)=|1-|1-3x||$ for $0\le x\le 1$. Then we can apply
Lemmas \ref{lem:tot-trans-entropy} and \ref{lem:pure-mix-entropy} to finish the proof.
\end{proof}

Let $T$ be a tree. We say that a point $p\in T$ is a \emph{central root} of $T$ if
both connected components of $T\setminus\{p\}$ are homeomorphic to each other. The full binary tree $B_n$ with $2^n$ endpoints can be defined inductively. Let $B_1=[0,1]$. Note that $1/2\in [0,1]$
is the central root of $B_1$.
Assume that we have defined
$B_n$ and $z_n\in B_n$ is a central root of $B_n$. Let $T'=(B_n,z_n) \wedge ([0,1],0)$ 
and let $z_{n+1}$ denote the endpoint $1\in T'$. We define $B_{n+1}= (T',z_{n+1})^{\wedge 2}$.
Clearly, $z_{n+1}\in B_{n+1}$ is a central root of $B_{n+1}$ and thus $B_{n+1}$ has $2^{n+1}$ endpoints.

\begin{lemma}\label{lem:binary-tree-entropy}
Let $n\ge 1$ and let $B_n$ be a complete binary tree with $2^n$ endpoints. Then for
every $\eps>0$, there is
\begin{enumerate}
\item an exact Markov map $F_\eps\colon B_n\defmapsto B_n$ such that
$(\log 3)/2^n\le h(F_\eps)<(\log 3)/2^n+\eps$ and there is a
fixed point of $F_\eps$ which is a central root for $B_n$.
\item a pure mixing map $G_\eps\colon B_n\defmapsto B_n$ such that
$(\log 3)/2^n\le h(G_\eps)<(\log 3)/2^n+\eps$.
\end{enumerate}
Moreover, all endpoints of $B_n$ form a single periodic orbit of $F_\eps$ and $G_\eps$.
\end{lemma}
\begin{proof}
First note that the second part of the theorem
follows from the first and Lemma \ref{lem:pure-mix-entropy}.
We will prove the first part by induction on $n$.
For $n=1$, note that $B_1=[0,1]$, and
consider the piecewise linear Markov map $f\colon [0,1]\defmapsto [0,1]$ given by
\[f(x)=
\begin{cases}
1-3x,&\text{for }0\le x<1/6,\\
3x,&\text{for }1/6\le x<1/3,\\
2-3x,&\text{for }1/3\le x<1/2,\\
1-x,&\text{for }1/2\le x\le 1.
\end{cases}
\]
It is clear that $f$ is transitive, but not totally transitive, and
$h(f)=\log\sqrt{3}$. Given any $\eps>0$ we can apply Lemma \ref{lem:tot-trans-entropy}
to get a piecewise linear Markov and exact map $F_\eps\colon B_1\defmapsto B_1$.
Note that $z=1/2$ is a
fixed point of $F_\eps$ which is a central root for $B_2$ and the endpoints of $[0,1]$
form a two cycle for $F_\eps$.

Now assume that the theorem holds for $n\ge 1$, fix an $\eps>0$,
and let $G\colon B_n\defmapsto B_n$ be an exact map provided by
induction hypothesis for $\eps/3$. Let $z$ be a fixed point of
$G$ which is a central root for $B_n$. Apply Theorem
\ref{thm:edge-adding} to obtain an exact piecewise linear
Markov map $G'\colon T'\defmapsto T'$ with $h(G')<h(G)+\eps/3$,
where  $T'$ is $B_n$ with an arc
$[z,z_0]$ attached to $z$, that is, $T'=(B_n,z)\wedge([z,z_0],z)$.
Note that the $2^n$ endpoints of $T'$ other than $z_0$ form a cycle for $G'$
and $z$ and $z_0$ are fixed
for $G'$. Observe that the $2$-nd wedge power $(T',z_0)^{\wedge 2}$
is just $B_{n+1}$. Now we can apply Lemma \ref{lem:wedge-entropy} to
get a transitive map $G''\colon B_{n+1}\defmapsto B_{n+1}$ with
$h(G'')=h(G')$. Applying Lemma \ref{lem:tot-trans-entropy} to
$G''$ and $\eps/3$ we get the map $F_\eps\colon B_{n+1}\defmapsto B_{n+1}$
with all desired properties.
\end{proof}

A $\sigma$-graph, $\theta$-graph, 8-graph 
are spaces homeomorphic to the symbol representing the Greek letter sigma, theta,
and the figure eight, respectively. A dumbbell is a graph homeomorphic to the
following subset of a complex plane $\mathbb{C}$: $C_{-2}\cup I\cup C_2$, where
$C_{\omega}=\{z\in\mathbb{C}:|z-\omega|=1$\} and $I$ is a line segment joining $z=-1$
with $z=1$.

\begin{theorem}Let $\pure{G}$ denote the family of all pure mixing maps of a given
graph $G$.
\begin{enumerate}
\item \label{ent-1} If $T_n$ is an $n$-star, $n\ge 2$, then
\[
\inf(h(\pure{T_n}))=\log 3 / n.
\]
\item \label{ent-2}
If $B_n$ is a full binary tree with $2^n$ endpoints, $n\ge 1$, then
\[
\inf(h(\pure{B_n}))=\log 3 / 2^n.
\]
\item \label{ent-3}
If $G_\sigma$ is a sigma graph, then
\[
\inf(h(\pure{G_\sigma}))=\log 3 / 2.
\]
\item \label{ent-4}
If $G_\theta$ is a theta graph, then
\[
\log 3/ 4 \le \inf(h(\pure{G_\theta})) \le \log 3/3.
\]
\item \label{ent-5}
Let $G_8$ be a figure-eight graph, and $G_d$ be the dumbbell graph, then
\[
\inf(h(\pure{G_8}))=\inf(h(\pure{G_d}))=\log 3 / 4.
\]
\end{enumerate}
\end{theorem}
\begin{proof}Part \eqref{ent-1}  and \eqref{ent-2} follow
from Lemmas~\ref{lem:n-star-entropy} and Lemma \ref{lem:binary-tree-entropy},
respectively.

To prove \eqref{ent-3} we fix $\eps>0$ and take
an exact map of the interval with the endpoints forming a
cycle of length two and entropy smaller than $\log\sqrt{3}+\eps/2$. By
Theorem~\ref{thm:edge-adding} there is an exact Markov map on the $3$-star,
with the entropy smaller than $\log\sqrt{3}+\eps$ and two out of three
endpoints of the $3$-star form a cycle for that map. By
Lemma \ref{lem:pure-mix-entropy} we can find a pure mixing map of the
$3$-star with the entropy smaller than $\log\sqrt{3}+\eps$ and two out of three
endpoints of the $3$-star still form a cycle which is inaccessible for that map.
Identifying those two endpoints we get a pure mixing map of the sigma graph
with topological entropy at most $\log\sqrt{3}+\eps$. On the other hand
it is easy to see that a pure mixing map of the sigma graph can have at most
two inaccessible sides since by Corollary~\ref{cor:sides-function} they have to form a cycle.Therefore its topological entropy is greater than
$\log\sqrt{3}$.

To see \eqref{ent-4} fix $\eps>0$ and
take a pure mixing map of the $3$-star $T_3$ with topological entropy
smaller than $\log\sqrt{3}/3+\eps$ for which endpoints of $T_3$ form an
inaccessible three cycle. Then we identify these endpoints to get a theta graph,
and the proof of the upper bound for the infimum is complete. The lower
bound comes from Theorem \ref{thm:3-horseshoe-kappa}.

The last point,
\eqref{ent-5} follows from Lemma \ref{lem:binary-tree-entropy} and
\ref{lem:n-star-entropy}, respectively.
To see this observe that identifying endpoints in the binary tree $B_2$ or $4$-star in the
appropriate way we get the dumbbell graph, and the figure-eight graph, respectively.
\end{proof}

We conjecture that the upper bound in the \eqref{ent-4} above is actually the
infimum, that is,  $\inf(h(\pure{G_\theta})) = \log 3/3$.
Note that the bound from Corollary \ref{cor:inf-pure} is not always
the best possible. It is an interesting question to find the formula
for $\inf(h(\pure{G}))$ depending on the combinatorial structure of $G$.




\section{Mixing implies specification property for graph maps}
In this section we present an alternative proof of the fact that every mixing
graph map has the specification property. The result was originally proved by
A.~Blokh. Our approach extends ideas of the proof presented by J.~Buzzi in
\cite{Buzzi} in the context of compact interval.
In order to carry out with the demonstration, we recall some terminology.
\begin{definition}
Let $n>0$ be an integer and let $\eps>0$. The \emph{closed Bowen ball} is the
set
\[
B_n(x,\eps)=\set{y\in G \; : \; \rho(f^i(x),f^i(y))\le\eps \textrm{ for } i=0,\ldots, n}.
\]
By $B_n'(x,\eps)$ we denote the connected component of $B_n(x,\eps)$ containing $x$.
\end{definition}

\begin{lemma}\label{lem:buzzi-main}
Let $f$ be a mixing map of a graph $G$. Assume that $\alpha,\eps>0$ are positive real
numbers. Then there are an integer $N=N(\alpha,\eps)>0$ and positive real number
$\delta=\delta(\eps)>0$ such that for every $y\in G$ and every integer $n\geq 0$
there is a point $z=z(y,\eps,n)\in G$ for which
\[
z\in B_n(y,\eps)
\quad\text{and}\quad
\ball(z,\delta)\subset \left(f^k (\ball(x,\alpha))\right)\setminus \left(\ball(\Nacc(f),\delta)\right)
\]
hold for every $x\in G$ and $k\ge N$.
\end{lemma}
\begin{proof}Fix $\alpha,\eps>0$. First, assume that $f$ is exact. It is easy to see
that for each $\delta>0$ there is an integer $N>0$ such that $f^k(B(x,\alpha))=G$ for
each $x\in G$ and $k\ge N$. Set $z=y$ and the proof for the first case is finished.

Consider the second case, when $f$ is pure mixing. By the uniform continuity and
compactness, it is enough to prove the assertion of the Lemma for a map $f^m$ for some
$m>0$. Using Theorem~\ref{thm:structure} we can find the integer $m$, graph $G'$, and
pure mixing map $g\colon G'\defmapsto G'$, which is an extension of $f^m$ and
$g^{-1}(e)=\{e\}$ for each $e\in \Nacc(g)$. It is sufficient to prove that the Lemma
holds for $g$. To simplify further our notation we can assume that $g$ has the unique
inaccessible endpoint $e_0$. The other cases can be handled analogously.
Using our convention $(C)$ we isometrically identify the edge containing $e_0$ with an
interval $[0,a]$, with $e_0=0$ and $a>0$.
We fix $\eps, \alpha>0$ and assume that $\eps<a$. By continuity and
Theorem~\ref{thm:3-horseshoe}, we can find fixed points $0<p<p'<\eps$ such that
$g([0,p'])\subset [0,\eps]$ and $g([0,p])\subset [0,p']$. Let $\delta=p/2$.
Theorem~\ref{thm:nacc}\eqref{con:nacc1} allows us to find an integer
$N>0$ such that $G'\setminus [0,\delta]\subset g^k(J)$ for
each subgraph $J$ of $G'$ with $\diam J\ge \alpha$ and each $k\ge N$. In particular,
$G'\setminus [0,\delta]\subset g^k(B(x,\alpha))$ for every $x\in G'$ and $k\ge N$.
Fix $y\in G'$ and let us agree that $\min \emptyset=\infty$. Consider
\[m=\min\{k:g^k(y)\in G'\setminus [0,p]\}.\]
We have the following cases:
\begin{description}
\item[Case 1] $m=0$. Then we set $z=y$,
\item[Case 2] $0<m<\infty$. Observe that $g^k(y)\in [0,p]$ for $k=0,1,\ldots m-1$.
Furthermore, $g^m(y)\in [p,p']$. But $g([p,p'])\subset [0,a]$ and since $p,p'$ are
fixed points for $g$, we see that $[p,p']$ is $g^m$-covered by itself.
Therefore, we can find $z\in [p,p']$ such that $g^m(z)=g^m(y)$.
\item[Case 3] $m=\infty$. Then it is enough to take $z=p$.
\end{description}
Clearly, in any case $z\in B_n(y,\eps)$ for all $n\ge 0$ and
$\ball(\Nacc(g),\delta)\cap B(z,\delta)=\emptyset$.
\end{proof}

%
%

\begin{lemma}\label{lem:geometric4}
Let $\alpha>0$. Then there is a constant $\gamma=\gamma(\alpha)$ such that for
every map $g\colon G\defmapsto G$ if $J\subset G$ is a star with $\diam g(J)\ge\alpha$,
then there is a free arc $J' \subset J$ which contains two subarcs $J_1,J_2$ with at most
one common point such that $g(J_i)$ is also a free arc $g$-covered by $J_i$ with
$\diam g(J_i)\ge \gamma$ for $i=1,2$.
\end{lemma}
\begin{proof}
By Lemma~\ref{lem:geometric3} there is a constant $\xi=\xi(\alpha)$ such that $g(J)$
contains a free arc $K$ with $\diam K\ge \xi$. We conclude from
Lemma~\ref{lem:covering}\eqref{lem:covering-star-image} that $K$ is contained in
$f(E_1\cup E_2)$ where $E_1,\,E_2$ are edges of $S$. Clearly, one of these edges, say
$E_1$, must $g$-cover a free arc $K'$ with $\diam K'\ge \xi/3$. We set $\gamma=\xi/6$
and write $K'=K_1\cup K_2$, where $K_1,\,K_2$ are free arcs with at most one common
point, and $\diam K_i\ge \gamma$. Now, an application of
Lemma~\ref{lem:covering}\eqref{lem:covering:Al2} finishes the proof.
\end{proof}

\begin{lemma}\label{lem:local-expansivity}
Let $f\colon G\defmapsto G$ be a mixing graph map. If $0<\eps<\frac{1}{2} \diam G$
and $\delta>0$ then there is an $N=N(\eps,\delta)>0$ such that $B_n'(x,\eps)\subset
B(x,\delta)$ for all $x\in G$ and all $n\ge N$.
\end{lemma}
\begin{proof}
If the conclusion of the Lemma does not hold, then, using compactness
of $G$, we could find a point $x\in G$ and $\eps>0$ such that the set
\[B=\bigcap_{k=0}^\infty B'_k(x,\eps)\] would have a non-empty interior.
Then $\diam f^n(B) \le 2\eps <\diam G$ for all $n$ contradicting topological
mixing.
\end{proof}


\begin{lemma}\label{lem:incompressibility}
If $f\colon G\defmapsto G$ is a mixing graph then for every $\eps>0$ there is a constant
$\zeta=\zeta(\eps)$ such that
\[
0<\zeta \le \diam f^n (\ball'_n(x,\eps))
\]
for every $n$ and $x\in G$.
\end{lemma}

\begin{proof}Fix $\eps >0$. Without loss of generality we may assume
that $\eps<1/2 \diam G$. Let $\mathcal{K}(\eps)$ be a family of
subgraphs of $G$ with diameter at least $\eps$. Then $\mathcal{K}(\eps)$
is a closed subset of a hyperspace of subcontinua of $G$. Moreover,
$\Phi_n=\diam\circ f^n$ is a continuous function on $\mathcal{K}(\eps)$
for all $n\ge 0$. Observe that there is an $N$ such that
$\diam f^n(J)\ge 1/2 \diam G$ for all $J\in\mathcal{K}(\eps)$, and
$n\ge N$. Therefore
\[
\beta(\eps):=\inf_{n\ge 0} \{\diam f^n(J): J\in \mathcal{K}(\eps)\} =
\min_{0\le j\le N}\min \Phi_j (\mathcal{K}(\eps)) >0,
\]
as no subgraph of $G$ can be mapped by any $f^n$ onto a point.

Fix $x\in G$. We claim that $\diam f^n(\ball'_n(x,\eps)) \ge \beta(\eps)$ defined above.
We have $f^{n+1}(\ball'_{n+1}(x,\eps))\subset f(f^{n}(\ball'_{n}(x,\eps)))$ for all $n$.
Furthermore, note that:
\begin{description}
\item[($\star$)]
If $f^{n+1}(\ball_{n+1}'(x,\eps))\neq f(f^n(\ball_{n}'(x,\eps)))$ then there is
$y\in B_{n}'(x,\eps)$ with
\[
\rho(f^{n+1}(x),f^{n+1}(y))\geq \eps,
\]
hence $\diam f^{n+1}(B_{n+1}'(x,\eps))\geq \eps$.
\item[($\star$$\star$)]
If $f^{n+1}(\ball_{n+1}'(x,\eps))= f(f^n(\ball_{n}'(x,\eps)))$ and
$\diam f^{n}(B_{n}'(x,\eps))\geq \eps$ then
\[
\diam f^{n+1}(\ball_{n+1}'(x,\eps))\ge \beta(\eps)
\]
by the definition of $\beta(\eps)$.
\end{description}
Applying $(\star)$ or $(\star\star)$, accordingly, we get:
\begin{description}\item[($\star$$\star$$\star$)]
If $\diam f^{n}(B_{n}'(x,\eps))\geq \eps$ for some $n\geq 0$ then
$\diam f^{n+k}(\ball_{n}'(x,\eps)))\ge \beta(\eps)$ for every $k\geq 0$ by the definition of $\beta(\eps)$.
\end{description}
Now, $\ball(x,\eps)=\ball'_0(x,\eps)$, and we proceed by induction.
\end{proof}

The following definition was introduced by Bowen in \cite{BowenSP}.

\begin{definition}
We say that a continuous map $f\colon X\ra X$ acting on a compact metric space $(X,d)$
has the \emph{specification property} if for every $\eps> 0$ there exists an integer
$M = M(\eps)$ such that for any $s > 1$, for any $s$ points $x_1,x_2,\ldots,x_s\in X$,
for any integers $a_1\leq b_1<a_2\leq b_2<\ldots<a_s\leq b_s$ with $a_i-b_{i-1}\geq M$,
for $2 \leq i \leq s$ and for any integer $p$ with $p \geq M +b_s -a_1$ there exists a
point $x \in X$ with $f^p(x) = x$ such that $d(f^n(x), f^n(x_i))<\eps$ for
$a_i \leq n \leq b_i$, and $1 \leq i \leq s$.
\end{definition}

\begin{remark}
Without loss of generality we can take $a_1=0$ and $p=M+b_s$ in the above definition.
\end{remark}
\begin{theorem}[Blokh]\label{thm:Blokh}
Every mixing graph map has the specification property.
\end{theorem}
\begin{proof}
Let $f\colon G\defmapsto G$ be a mixing graph map and fix $\eps>0$. Our first task
is to find an suitable integer $N>0$ as in the definition of specification.
Without lost of generality we may assume that $\eps<(1/2)\diam G$. Set
$\alpha=\diam I_U$, where $I_U$ is an universal arc for $f$. Let
$N_1=N(\eps/2,\alpha/2)$ and $\delta=\delta(\eps/2)$ be given by
Lemma \ref{lem:buzzi-main} for $\eps/2$ and $\alpha$ as above. We may also
assume that the open ball $\ball(x,\delta)$ is a canonical neighborhood of
$x$ for each $x\in G$. We plug $\delta$, and $\eps/2$ into
Lemma~\ref{lem:local-expansivity} to get an $N_2=N(\eps/2,\delta)$ such that
$B_n'(x,\eps/2)\subset\ball(x,\delta)$ for all $x\in G$ and all $n\ge N_2$.
As a direct consequence of Lemma~\ref{lem:incompressibility} and
Lemma~\ref{lem:geometric4} we can find a constant $\beta>0$ such that for every
$x\in G$ and every $n>0$ the following condition holds
\begin{description}
\item[($\star$)] there is a free arc $J'\subset B_n'(x,\eps/2)$ containing two
free arcs $J_1$, and $J_2$ with at most one common point and two free arcs
$K_1$ and $K_2$ with $\diam K_i\ge \beta$ for $i=1,2$ and such that $K_i$ is
$f^n$ covered by $J_i$ for $i=1,2$.
\end{description}
For $\beta>0$ as defined above we can find an $N_3>0$ such that if $n\ge N_3$
then $I_U$ is $f^n$ covered by each closed interval $K$ with $\diam K>\beta$.

We claim that $N = N_1+N_2+N_3$ will fulfill the definition of the
specification property. For the proof of our claim we assume to simplify the
notation that $s=2$ and we choose any $x_1,x_2\in G$, and integers
$0=a_1\leq b_1<a_2\leq b_2$, with $a_2-b_1\ge N$ . Finally, we fix any
$p\ge b_2+N$. For $i=1,2$ let $y_i=f^{a_i}(x_i)$ and $n(i)=b_{i}-a_{i}+N_2$.
Let $z_i=z(y_i,\eps/2,n(i))\in B_{n(i)}(y_i,\eps/2)$ be provided by
Lemma~\ref{lem:buzzi-main}. By our choice of $N_2$ we conclude that
$B_i=B'_{n(i)}(z_i,\eps/2)\subset \ball(z_i,\delta)$ for $i=1,2$. Let us
denote the free arcs constructed in $B_i$ in $(\star)$ above by $J^i_1$ and
$J^i_2$. By Lemma~\ref{lem:buzzi-main} we have $\ball(z_i,\delta)\subset f^n(I_U)$
for each $i=1,2$ and $k\ge N_1$. Applying Lemma~\ref{lem:covering}\eqref{lem:covering:Al3}
we see that for $i=1,2$ and $k$ as above $I_U$ must $f^k$ cover at least one
free arc $I_i(k)\in\{J^i_1,J^i_2\}$. Let $I_1=I_1(k(1))$, where
$k(1)=a_2-b_1-N_2+N_3$, and $I_2=I_2(k(2))$, where $k(2)=p-(b_2+N_2+N_3)$.
Appealing again to the condition $(\star)$ we see that each $I_i$ covers
through $f^{n(i)}$ an interval $K_i$ with $\diam K_i\ge \beta$. This in turn
implies that each $K_i$ must $f^{N_3}$-cover $I_U$. In conclusion, we get
\[
I_1\stackrel{f^{n(1)}}{\implies} K_1 \stackrel{f^{N_3}}{\implies}  I_U
   \stackrel{f^{k(1)}}{\implies} I_2 \stackrel{f^{n(2)}}{\implies} K_2
   \stackrel{f^{N_3}}{\implies}  I_U \stackrel{f^{k(2)}}{\implies} I_1
\]
where $I\stackrel{f}{\implies} K$ denotes here that $I$ $f$-covers $K$.
It follows that $I_1$ is $f^p$ covered by itself, since $p=n(1)+k(1)+n(2)+k(2)+2N_3$.
Therefore there is a $p$-periodic point $q\in I_1$ such that
$q\in B'_{n(1)}(z_1,\eps/2)$ and $z_1\in B_{n(1)}(y_1,\eps/2)$. Moreover,
$r=f^{a_2}(q)\in I_2$, hence $r\in B'_{n(2)}(z_2,\eps/2)$ and
$z_2\in B_{n(2)}(y_2,\eps/2)$. This finishes the proof.
\end{proof}

\section*{Acknowledgements}
This work was supported by the Polish Ministry of Science and Higher Education from sources for science in the years 2010-2011, grant no. IP2010~029570.

\bibliographystyle{amsplain}
\bibliography{hko_bib_20110316}
\end{document}